\documentclass[10pt]{aims}
\usepackage{amsmath, amssymb, mathrsfs}
  \usepackage{paralist}
  \usepackage{graphics} 
  \usepackage{epsfig} 
\usepackage{graphicx} 
 \usepackage{epstopdf}
 \usepackage[colorlinks=true]{hyperref}
 \hypersetup{urlcolor=blue, citecolor=blue}
 \usepackage[toc,page]{appendix}
\usepackage{chngcntr}
\usepackage{hyperref}

\usepackage{titlesec}
\titleformat{\section}{\Large\bfseries}{\thesection.}{4pt}{}
\titleformat{\subsection}{\large\bfseries}{\thesection.\arabic{subsection}.}{4pt}{}
\titleformat{\subsubsection}{\bfseries}{\thesection.\arabic{subsection}.\arabic{subsubsection}.}{4pt}{}
\titleformat*{\paragraph}{\bfseries}
\titleformat*{\subparagraph}{\bfseries}
\usepackage[margin=1in]{geometry}

  \def\currentyear{}
   \def\currentmonth{}
    \def\ppages{}
     \def\DOI{}

\newtheorem{theorem}{Theorem}[section]
\newtheorem{corollary}[theorem]{Corollary}

\newtheorem{lemma}[theorem]{Lemma}
\newtheorem{proposition}[theorem]{Proposition}
\theoremstyle{definition}
\newtheorem{definition}[theorem]{Definition}
\newtheorem{remark}[theorem]{Remark}

\numberwithin{equation}{section}

\title[The profile  for the imaginary part of a blowup solution]{Profile for the imaginary part of  a  blowup solution for a  complex-valued semilinear  Heat
 Equation}

\author[G. K. Duong]{}
\subjclass{Primary: 35K55, 35K57 35K50, 35B44; Secondary: 35K50, 35B40.}
 \keywords{Blowup solution, Blowup profile, Stability, Semilinear complex  heat equation, non variation heat equation}

\thanks{\today}
\begin{document}
\maketitle

\centerline{Giao Ky Duong \footnote{ G. K. Duong  is supported   by the project INSPIRE. This project has received funding from the European Union’s Horizon 2020 research and innovation programme under the Marie Sk\l odowska-Curie grant agreement No 665850.}  }
\medskip
{\footnotesize
  \centerline{ Universit\'e Paris 13, Sorbonne Paris Cit\'e, LAGA,  CNRS (UMR 7539), F-93430, Villetaneuse, France.} 
}

\bigskip
\begin{center}\thanks{\today}\end{center}

\begin{abstract} 
In this paper, we  consider  the following complex-valued semilinear   heat equation
\begin{eqnarray*}
\partial_t u =  \Delta u + u^p, u \in \mathbb{C},
\end{eqnarray*}
in the whole space $\mathbb{R}^n$, where $ p \in \mathbb{N}, p \geq 2$.
 We aim at   constructing  for this equation a complex solution $u = u_1 + i u_2$, which blows up in finite 
  time $T$ and only at one blowup point $a$, with  the following estimates for the final profile 
\begin{eqnarray*}
u(x,T)  &\sim &  \left[ \frac{(p-1)^2 |x-a|^2}{ 8 p |\ln|x-a||}\right]^{-\frac{1}{p-1}},\\
u_2(x,T)  &\sim &   \frac{2 p}{(p-1)^2}  \left[ \frac{ (p-1)^2|x-a|^2}{ 8p |\ln|x-a||}\right]^{-\frac{1}{p-1}}\frac{1}{ |\ln|x-a||} , \text{ as } x \to a.
\end{eqnarray*}
Note that the imaginary part is non-zero and that it blows up also at point $a$. 
 Our method  relies on  two main  arguments:  the reduction of
 the  problem  to a  finite  dimensional one  and a topological argument based
   on the index theory to get the conclusion.  Up to our knowledge,  this is the first time where  the blowup behavior of the imaginary part is derived in multi-dimension.
\end{abstract}

\maketitle
\section{Introduction}
In this work,   we are interested in the following complex-valued semilinear  heat equation
\begin{equation}\label{equ:problem}
\left\{
\begin{array}{rcl}
\partial_t u &=& \Delta u +  F(u), t \in [0,T), \\[0.2cm]
u(0) &=& u_0 \in L^\infty,
\end{array}
\right.
\end{equation}
where  $F(u) = u^p$ and $u(t): \mathbb{R}^n \to \mathbb{C}$,  $L^\infty := L^\infty(\mathbb{R}^n, \mathbb{C})$,  $p   > 1$.  Though our results hold only when $p \in \mathbb{N}$ (see Theorem \ref{Theorem-profile-complex} below), we keep $p\in \mathbb{R}$ in the  introduction, in order to broaden the discussion.

In particular,   when $p =2$,  model \eqref{equ:problem}  evidently  becomes
\begin{equation}\label{equation-problem-p=2}
\left\{
\begin{array}{rcl}
\partial_t u &=& \Delta u +  u^2, t \in [0,T), \\[0.2cm]
u(0) &=& u_0 \in L^\infty.
\end{array}
\right.
\end{equation}

We     remark that equation  \eqref{equation-problem-p=2} is rigidly  related to the viscous Constantin-Lax-Majda equation with a viscosity term,
 which is a one dimensional model for the  vorticity  equation in fluids. 
 The readers can    see more  in some of  the typical works:  Constantin, Lax, Majda \cite{CLMCPAM1985}, Guo, Ninomiya and  Yanagida 
 in \cite{GNYTAMS2013},   Okamoto, Sakajo and Wunsch  \cite{OSWnon2008}, Sakajo in  \cite{SMAthScitokyo2003} and \cite{Snon2003}, Schochet \cite{SCPAM1986}. 
The  local  Cauchy problem for model \eqref{equ:problem} can be   solved (locally in time)
  in $L^{\infty}(\mathbb{R}^n,\mathbb{C})$  if    $p$ is integer,  by using a  fixed-point argument. However,  when  $p $ is not  integer,  the  local  Cauchy  problem has not   been  sloved yet, up to our knowledge. This probably  comes from   the discontinuity of    $F(u) $  on $\{u   \in \mathbb{R}^*_-\}$.  In addition to  that,  let us remark that equation 
\eqref{equ:problem} has the following family of space  independent solutions:
\begin{equation}\label{solu-indipendent-p-notin-Q}
u_k (t) = \kappa e^{i \frac{2k\pi}{p-1}} \left( T -t\right)^{-\frac{1}{p-1}},\text{ for any } k \in \mathbb{Z},  
\end{equation} 
where $\kappa = (p-1)^{-\frac{1}{p-1}}$. 

If $p \in \mathbb{Q}$, this makes a finite number of solutions.

If $p \notin \mathbb{Q},$  then the set
\begin{equation}\label{rescaling-u_k}
\left\{ u_k(t) \frac{(T - t)^{\frac{1}{p-1}}}{\kappa} \left| \right. \quad k \in \mathbb{Z}    \right\},
\end{equation}
is  countable  and dense in the unit circle  of $\mathbb{C}$.

\noindent This latter case ($p \notin \mathbb{Q}$), is somehow  intermediate between 
the case $(p \in \mathbb{Q})$ and the case of the twin PDE 
\begin{equation}\label{heat-real-|u|p-1u}
\partial_t u = \Delta u + |u|^{p-1} u,
\end{equation}
which admits the following family of space independent solutions 
 $$ u_{\theta} (t) = \kappa e^{i \theta} (T -t)^{-\frac{1}{p-1}},$$
for any $\theta \in \mathbb{R}$, which  turns to be infinite  and covers all the unit 
circle, after rescaling as in \eqref{rescaling-u_k}.  In fact,   equation \eqref{heat-real-|u|p-1u}  is certainly much easier than equation \eqref{equ:problem}. As a mater of  fact, it reduces to the scalar case thanks to a modulation technique, as Filippas and Merle did in \cite{FMjde95}. 
\bigskip

Since the Cauchy  problem for equation \eqref{equ:problem} is already  hard when $p \notin \mathbb{N}$, and  given that we are more interested in the asymptotic blowup behavior, rather than the well-posedness issue, we will focus in our paper on the case $p \in \mathbb{N}$. 
In this case, from the Cauchy theory, the solution of equation \eqref{equ:problem}   either    exists globally or blows up in finite time.     Let us  recall that the solution $u(t) = u_1  (t) + i u_2(t)$ blows up in finite time $T < +\infty$ if and only if it exists for all $t \in [0,T)$ and 
$$\limsup_{t \to T} \{\|u_1(t)\|_{L^{\infty}} + \| u_2(t)\|_{L^{\infty}}\} \to +\infty.$$  
 If $u$ blows up in finite time $T$, a point $a \in \mathbb{R}^n$ is called a  blowup point if and only if 
there exists a sequence $\{(a_j, t_j)\} \to (a,T)$ as $j \to +\infty$ such that 
$$ |u_1(a_j,t_j)| + |u_2(a_j,t_j)| \to +\infty \text{ as } j \to +\infty.$$

\bigskip
 The blowup phenomena  occur  for  evolution equations in general,  and  in semilinear heat equations in particular.  Accordingly, an interesting question is  to construct for those equations   a solution which   blows up in finite time and  to describe its blowup behavior. These questions are being  studied by many authors in the world. Let us  recall some blowup results connected to our equation:
 
\bigskip
\textbf{ $(i)$ The real case:}     Bricmont and
 Kupiainen \cite{BKnon94}  constructed a real positive solution  to \eqref{equ:problem} for all $p> 1$,  which blows up in finite time $T$, only at the origin and they also gave the  profile of  the solution such that
 $$  \left\|  (T-t)^{\frac{1}{p-1}}  u (x,t)  - f_0 \left( \frac{x}{\sqrt{(T-t)|\ln (T - t)|} } \right) \right\|_{L^{\infty}(\mathbb{R}^n)} \leq \frac{C}{ 1 + \sqrt{|\ln (T-t)|}},$$
 where the profile $f_0 $ is defined  as follows
 \begin{equation}\label{defini-f-0}
 f_0 (z) = \left( p-1 + \frac{(p-1)^2 |z|^2}{4p}\right)^{-\frac{1}{p-1}}.
 \end{equation}
 
 \bigskip
 \noindent In addition to  that,   with a  different method, Herrero and  Vel\'azquez in \cite{HVcpde92}   obtained  the same  result. Later,  in \cite{MZdm97}  Merle and Zaag simplified  the proof  of  \cite{BKnon94} and    proposed the following two-step method (see also the note   \cite{MZasp96}):
    \begin{itemize}
    \item[-] Reduction of the  infinite dimensional problem to a finite dimensional one.
    \item[-]   Solution of the finite   dimensional problem thanks to a topological  argument based  on  Index theory.
    \end{itemize}
 We would like to mention that this   method has  been   successful  in  various  situations  such as the work of Tayachi and Zaag \cite{TZpre15}, 	and also  the works of Ghoul, Nguyen and Zaag in \cite{GNZpre16a}, \cite{GNZpre16b}, and \cite{GNZsysparabolic2016}. In those papers, the considered equations were scale invariant;   this  property was believed to be essential   for the construction. Fortunately, with the work of  Ebde and Zaag  \cite{EZsema11}  for the following equation 
$$   \partial_t u =    \Delta u + |u|^{p-1} u + f (u, \nabla u),$$
 where 
 $$ |f(u,\nabla u )  | \leq C ( 1 + |u|^q   + | \nabla u|^{q'}) \text{ with }  q < 	p , q' < \frac{2p}{p+1},$$
 that belief  was proved to be wrong.
 
\noindent   Going in the same direction as  \cite{EZsema11},  Nguyen and Zaag in \cite{NZens16}, have achieved   the construction  with a   stronger perturbation  
 $$ \partial_t u  = \Delta u  + | u|^{p-1} u + \frac{\mu  |u|^{p - 1} u }{ \ln ^a (2 + u^2)} ,$$    
   where  $\mu \in \mathbb{R},  a > 0 $.  Though the results of  \cite{EZsema11}  and \cite{NZens16}  show that the invariance  under  dilations of the equation  in  not necessary in the construction method, we might  think that the construction of \cite{EZsema11}  and \cite{NZens16}  works  because the authors adopt a  perturbative  method around the pure power case $ F(u) =   |u|^{p-1} u$. If this is true  with \cite{EZsema11}, it is not the case for \cite{NZens16}.  In order to totally prove that the construction does not need  the invariance by dilation,  Duong, Nguyen and Zaag  considered  in \cite{DNZtunisian-2017} the following  equation 
   $$ \partial_t u =  \Delta u  + |u|^{p - 1} u \ln ^\alpha ( 2 + u^2),$$
for some where $\alpha \in \mathbb{R} $ and $  p> 1$, where we have no invariance under dilation, not even for the main term on the nonlinearity. They   were   successful in constructing a stable blowup solution for that  equation. Following the above mentioned discussion,    that work  has to be considered  as a breakthrough.

\bigskip
 \noindent Let us mention that a classification  of  the blowup behavior of \eqref{equation-problem-p=2} was  made available    by many authors  such as   Herrero and Vel\'azquez in \cite{HVcpde92}  and Vel\'azquez in \cite{VELcpde92}, \cite{VELtams93}, \cite{VELiumj93} (see also  Zaag in \cite{ZAAcmp02} for some refinement). More precisely and  just to stay in one space dimension for simplicity, it is  proven in  \cite{HVcpde92}    that if $u$ a real solution of \eqref{equ:problem},  which blows up in finite time $T$  and $a$ is a given blowup point, then:
\begin{itemize}
\item[$A.$] Either 
$$\sup_{|x - a| \leq K \sqrt{(T - t) |\ln   (T - t)|} }  \left| \left( T - t\right)^{\frac{1}{p-1}} u(x,t)  - f_0\left( \frac{x - a}{\sqrt{(T - t) |\ln   (T - t)|}}\right)\right| \to 0 \text{ as } t \to T,$$
for any $K > 0$ where   $f_0 (z)  $ is defined in \eqref{defini-f-0}.
\item[$B.$] Or, there exist $m  \geq 2, m \in \mathbb{N}$ and $C_m > 0$ such that
$$\sup_{|x - a| \leq K (T - t)^{\frac{1}{ 2 m}} }  \left| \left( T - t\right)u(x,t)  - f_m\left( \frac{C_m(x - a)}{ (T - t)^{\frac{1}{2m}}}\right)\right| \to 0 \text{ as } t \to T,$$
for any $K > 0$, where $f_m (z)  =  (p-1  + |z|^{2m})^{-\frac{1}{p-1}}$.
\end{itemize} 

\medskip
\textbf{    $(ii)$  The complex case:}     	The blowup   question for the complex-valued parabolic equations has been  studied  intensively by many authors, in particular  for  the Complex    Ginzburg Landau (CGL) equation 
\begin{equation}\label{ginzburg-landau}
 \partial_t u =   (1 + i \beta ) \Delta u   + (1  + i \delta) | u|^{p-1} u  + \gamma u. 
\end{equation}
 This is the case of an   ealier work of  Zaag  in \cite{ZAAihn98}    for  equation \eqref{ginzburg-landau} when     $\beta = 0$ and $ \delta$ small enough. Later,  Masmoudi and Zaag  in  \cite{MZjfa08}  generalized the result of  \cite{ZAAihn98} and  constructed  a blowup solution  for \eqref{ginzburg-landau}  with    $ p - \delta^2 -  \beta \delta   -  \beta  \delta p   >0   $      such that the solution satisfies the following
\begin{eqnarray*}
\left\|  (T - t)^{\frac{1 + i \delta}{p-1}  }    | \ln (T -t)|^{-i \mu } u(x,t) -   \left(p-1  + \frac{b_{sub} |x|^2}{(T-t)|\ln (T -t)|} \right)^{- \frac{1 + i \delta}{p-1}} \right\|_{L^\infty}    \leq \frac{C}{1 + \sqrt{| \ln(T-t|}},
\end{eqnarray*}
  where 
  $$b_{sub}   =  \frac{(p-1)^2}{ 4 ( p - \delta^2 -  \beta \delta   -  \beta  \delta p   )} > 0 .$$
 Then,   Nouaili and Zaag  in \cite{NZ2017}  has constructed for \eqref{ginzburg-landau} (in case  the critical where  $\beta = 0$ and $ p = \delta^2 $) a  blowup solution
   satisfying 
   \begin{eqnarray*}
   \left\|  (T - t)^{\frac{1 + i \delta}{p-1}  }    | \ln (T -t)|^{-i \mu } u(x,t) -   \kappa^{-i \delta }\left(p-1  + \frac{b_{cri} |x|^2}{(T-t)|\ln (T -t)|^{\frac{1}{2}}} \right)^{- \frac{1 + i \delta}{p-1}} \right\|_{L^\infty}    \leq \frac{C}{1 +  |\ln(T-t)|^{\frac{1}{4}} },
   \end{eqnarray*}
   with
   $$ b_{cri}  =  \frac{(p-1)^2}{ 8 \sqrt{p(p+1)}}, \mu  = \frac{\delta}{8 b }.$$

\bigskip
\noindent
As for  equation \eqref{equation-problem-p=2},  there are many works  done in dimension one, such as the work of   Guo, Ninomiya, Shimojo and  Yanagida,  who   proved in \cite{GNYTAMS2013}  the following results (see Theorems 1.2, 1.3 and 1.5 in this work):

\medskip

$(i)$  \noindent \textit{ (A Fourier- based  blowup  crieterion). We assume that the  Fourier  transform  of  initial data  of \eqref{equation-problem-p=2} is real and positive, then the solution blows up in finite time.}
 
 \medskip
$(ii)$  \noindent \textit{ (A simultaneous blowup  criterion in dimension one)  If the initial data  $u^0  =  u_1^0  + u_2^0,$ satisfies 
 $$ u_1^0 \text{ is even },  u_2^0  \text{ is odd  with } u_2^0 > 0 \text{ for } x > 0.$$  
 Then,  the fact that the   blowup set is compact  implies that $u_1^0, u_2^0$ blow up simultaneously. }

 \medskip
 $(iii)$   \noindent \textit{  Assume that $u_0 = u_1^0 + i u_2^0$ satisfy
\begin{eqnarray*}
u_1^0, u_2 ^0 \in C^1 (\mathbb{R}^n), 0 \leq u_1^0  \leq 	M,  u_1^0 \neq  M, 0 < u_2^0 \leq L, 
\end{eqnarray*}
 \begin{eqnarray*}
 \lim_{|x| \to + \infty} u_1^0 (x)  = M      \text{ and } \lim_{|x| \to +\infty} u_2^0 = 0,
 \end{eqnarray*}
 for some  constant  $L, M $. Then, the solution $u= u_1 + i u_2$ of \eqref{equation-problem-p=2}, with initial data $u^0$, blows up  at time $T (M),$ with $u_2 (t) \not\equiv 0$ . Moreover,  the real part  $u_1(t)$ blows up only at  space  infinity  and  $u_2 (t)$  remains   bounded. }
 
 \medskip
 
\noindent   Still for equation   \eqref{equation-problem-p=2},    Nouaili and Zaag constructed  in \cite{NZCPDE2015}    a complex solution $ u = u_1+ i u_2 ,$  which blows up in finite time $T$  only at the  origin. Moreover, the solution  satisfies the following asymptotic behavior
$$\left\| (T -t) u (.,t)  - f \left( \frac{.}{ \sqrt{ (T - t) |\ln (T -t)|}}\right)  \right\|_{L^{\infty}}  \to 0  \text{ as } t \to T,$$
where     $f(z) =  \frac{1}{ 8  + |z|^2}$   and  the imaginary part satisfies the following astimate   for all $ K > 0$
\begin{equation}\label{asymptotic-imiginary-nouaili-zaag}
\sup_{|x| \leq K \sqrt{T -t}}  \left| (T -t) \tilde v (x, t)  - \frac{1}{|\ln (T -t)|^2} \sum_{j=1}^n C_j \left( \frac{x_j^2}{T - t}  - 2\right)  \right| \leq \frac{C(K)}{ | \ln (T - t)|^\alpha},
\end{equation}
for some   $(C_i)_i \neq (0,...,0)$ and $ 2 < \alpha < 2 + \eta, \eta $ small enough. Note that the real and the imaginary parts blow up simultaneously  at the origin. Note also that \cite{NZCPDE2015} leaves unanswered the question of the derivation of  the profile of the imaginary  part, and this is precisely  our aim in this paper, not only for equation \eqref{equation-problem-p=2}, but also for equation  \eqref{equ:problem} with  $ p \in \mathbb{N},  p \geq 2$. 

\medskip
\noindent Before stating our result (see Theorem \ref{Theorem-profile-complex}  below), we would  like to mention some classification results by Harada  for blowup solutions of \eqref{equation-problem-p=2}. As a matter of fact, in \cite{HaradaJFA2016}, he  classified all blowup solutions of \eqref{equation-problem-p=2} in dimension one,  under some reasonable assumption (see  \eqref{condition-Harada}, \eqref{harada-imaginary-condi}),    as follows (see Theorems 1.4, 1.5 and 1.6 in that work): 

\medskip
\noindent \textit{  Consider   $u =  u_1 + i u_2 $  a   blowup solution  of \eqref{equation-problem-p=2}   in one dimension space  with blowup time $T$ and blowup  point $\xi$  which satisfies 
\begin{equation}\label{condition-Harada}
 \sup_{0< t < T} (T-t) \| u(t)\|_{L^\infty} < +\infty.
\end{equation}
Assume  in addition that
\begin{equation}\label{harada-imaginary-condi}
 \lim_{s \to +\infty} \| w_2(s)  \|_{L^2_\rho(\mathbb{R})}   = 0, w_2 \not \equiv 0, 
\end{equation}
where  $\rho$ is defined  as follows 
 \begin{equation}\label{defi-rho-n=1}
 \rho (y) = \frac{e^{-\frac{y^2}{4}}}{\sqrt{4 \pi}},
 \end{equation}
and  $w_2$  is   defined  by the following change of variables (also  called similarity variables):
$$ w_1 (y,s) = (T-t) u_1( \xi + e^{-\frac{s}{2}}y, t ) \text{ and } w_2(y,s) =  (T-t) u_2( \xi + e^{-\frac{s}{2}}y, t ) , \text{ where  } t = T -e^{-s}.$$  
Then, one of the following cases  occurs
 \begin{eqnarray*}
 &(C_1)& \left\{  \begin{array}{rcl}
 w_1  &=&  1 - \frac{c_0}{s} h_2 + O (\frac{\ln s}{s^2} )  \text{ in } L^2_{\rho} (\mathbb{R}),\\[0.3cm]
 w_2 &= &c_2 s^{- m} e^{- \frac{(m-2)s}{2}} h_m + O \left(  s^{-(m+1)} e^{- \frac{(m-2)s}{2}} \ln s\right) \text{ in }  L^2_\rho(\mathbb{R}), m \geq 2.
\end{array}  \right.\\ 
 &(C_2)& \left\{ \begin{array}{rcl}
 u &=&  1 - c_1 e^{- (k-1)s} h_{2k} + O ( e^{- \frac{(2k - 1)s}{2}} )  \text{ in } L^2_{\rho} (\mathbb{R}),\\[0.3cm]
 v &=& c_2 e^{- \frac{(m-2)s}{2}} h_m + O \left(  e^{- \frac{(m-1)s}{2}} \right) \text{ in }  L^2_\rho(\mathbb{R}), k \geq 2, m \geq 2k.
\end{array}  \right. 
 \end{eqnarray*} 
 where $c_0  =  \frac{1}{8}   , c_1 > 0, c_2 \neq 0$ and    $\rho(y)$  is defined in   \eqref{defi-rho-n=1}
 and $h_j(y)$ is  a rescaled version of  the  Hermite polynomial of order $m^{th}$  defined as follows:
\begin{equation}\label{Hermite}
h_m(y)   = \sum_{j=0}^{\left[ \frac{m}{2}\right]} \frac{(-1)^jm! y^{m - 2j}}{j! ( m -2j)!}.
\end{equation}
 . }
 
\medskip 
 \noindent Besides that, Harada has   also  given a profile  to the solutions in similarity variables:
 
 \medskip
 \noindent \textit{There exist $\kappa, \sigma, c > 0$  such that 
 \begin{eqnarray}
 (C_1) &  \Rightarrow &  \left| u - \frac{1}{1 + c_0 s^{-1} h_2 }      \right|   + \left|s^{\frac{m}{2}} e^{\frac{(m-2)s}{2}} v - \frac{c_2 s^{-\frac{m}{2}} h_m}{(1 + c_0 s^{-1}h_2)^2} \right|  < c s^{-\kappa},\label{profile-real-hara}\\
 & \text{ for } & |y| \leq s^{(1 + \sigma)}.\nonumber\\
 (C_2) & \Rightarrow & \left| u - \frac{1}{ 1 +  c_1 e^{- (k-1 ) s } h_{2k} }   \right|  + \left| e^{\frac{(m - 2k)s}{2k} } v - \frac{c_2 e^{- \frac{(k-1)m s}{2k}}h_m}{ (1  + c_1 e^{-(k-1)s}h_{2k})^2}    \right|,\label{profile-imaginary-hara}\\
 &\text{ for } & |y| \leq e^{\frac{(k-1 + \sigma)s}{2k}}.\nonumber
\end{eqnarray}}

\bigskip
\noindent
Furthemore , he also gave the final blowup profiles

\bigskip
\noindent \textit{The blowup profile of $u= u_1 + i u_2$ is given by 
\begin{eqnarray*}
(C_1)  & \Rightarrow & \left\{ \begin{array}{rcl}
u_1(x,T) & = & \frac{2}{c_0} \left(\frac{|\ln |x||}{ x^2} \right) (1  + o (1)),\\[0.3cm]
u_2(x,T) &= & \frac{c_2 }{2^{m-2} (c_0 )^2} \left(\frac{x^{m-4}}{|\ln|x||^{m-2}} \right) (1 + o(1)),
\end{array} \right.\\[0.3cm]
(C_2) & \Rightarrow &  \left\{   \begin{array}{rcl}
u(x, T) &=&  \frac{1 + i c_1}{(c_1 - i c_2) } x^{-2k} (1 + o (1)), \\[0.3cm]
  \text{ if }  m & =   & 2k,\\[0.3cm]
  u_1(x,T) & = & (c_1 )^{-1} x^{-2k} (1 + o(1)) \text{ and }  u_2(x,T) =  \frac{c_2 }{ (c_1 )^2} x^{m - 4k} (1 + o (1)), \text{ if } m > 2k.
\end{array}  \right.
 \end{eqnarray*}}

 Then,  from  the work of   Nouaili and Zaag in \cite{NZCPDE2015}    and  Harada in  \cite{HaradaJFA2016} for  equation \eqref{equation-problem-p=2},   we  derive that the imaginary part $u_2$ also blows up  under some conditions, however, none of them was able to give    a global   profile (i.e. valid uniformly on $\mathbb{R}^n$, and not just on an expanding ball as  in \eqref{profile-real-hara} and \eqref{profile-imaginary-hara})  for the imaginary part. For that reason, our main motivation in this work is to give a sharp description for the profile of the imaginary part.   Our work is considered as an improvement of   Nouaili and Zaag  in \cite{NZCPDE2015} in dimension $n$, which  is valid not only for $p=2$, but also for any $p\geq 3, p \in \mathbb{N}$.  In particular, this   is the  first time we  give   the   profile for   the imiginary  part when  the solution blows up.  More precisely, we have the following Theorem:
\begin{theorem}[Existence of  a  blowup solution for \eqref{equ:problem} and a sharp discription of its profile]\label{Theorem-profile-complex}    For each  $p \geq 2, p  \in  \mathbb{N}$ and $  p_1 \in (0,1) $,  there exists $T_1 (p,p_1)> 0 $ such that for all $T \leq T_1,$ there exist   initial data  $u^0  = u^0_1  + i u^0_2,$   such that  equation \eqref{equ:problem} has  a  unique solution $u (x,t)$ for all $(x,t) \in \mathbb{R}^n \times [0,T)$  satisfying the following: 
\begin{itemize}
\item[$i)$] The solution $u$ blows up in finite time $T$   only at the  origin. Moreover, it  satisfies  the following estimates
\begin{equation}\label{esttima-theorem-profile-complex}
\left\| (T- t)^{\frac{1}{p-1}} u (x,t) - f_0 \left( \frac{x}{\sqrt{(T -t) |\ln(T -t)|}}  \right)  \right\|_{L^{\infty}(\mathbb{R}^n)}  \leq \frac{C}{ \sqrt{|\ln (T -t)|}},
\end{equation}
and
\begin{equation}\label{estima-the-imginary-part}
\left\| (T- t)^{\frac{1}{p-1}} |\ln (T-t)|  u_2(x,t) - g_0 \left( \frac{x}{\sqrt{(T -t) |\ln(T -t)|}}  \right)  \right\|_{L^{\infty}(\mathbb{R}^n)}  \leq \frac{C}{ |\ln (T -t)|^{\frac{p_1}{2}}},
\end{equation}
where $f_0 $ is defined in \eqref{defini-f-0} and  $g_0(z)$ is defined as follows
\begin{eqnarray}
\displaystyle g_0(z) &=& \frac{|z|^2}{\left(p-1  + \frac{(p-1)^2}{4p} |z|^2 \right)^{\frac{p}{p-1}}}\label{defini-g-0-z}.
\end{eqnarray} 
\item[$ii)$] There exists  a complex   function 
 $u^*(x) \in C^{2}(\mathbb{R}^n \backslash \{0\})$ such that
  $u(t) \to u^* = u_1^* + i u_2^*$  as  $t \to T$ uniformly on compact sets of 
  $\mathbb{R}^n \backslash \{0\}$ and  we have the following asymptotic expansions:
\begin{equation}\label{asymp-u-start-near-0-profile-complex}
u^*(x) \sim \left[ \frac{(p-1)^2 |x|^2}{ 8 p |\ln|x||}\right]^{-\frac{1}{p-1}}, \text{ as } x \to 0.
\end{equation}
and  
\begin{equation}\label{asymp-u-start-near-0-profile-complex-imaginary-part}
 u^*_2(x)  \sim \frac{2 p}{(p-1)^2}  \left[ \frac{ (p-1)^2|x|^2}{ 8p |\ln|x||}\right]^{-\frac{1}{p-1}}\frac{1}{ |\ln|x||} , \text{ as } x \to 0.
\end{equation}
\end{itemize}
\end{theorem}

\begin{remark}\label{remark-initial-data}
The initial data  $u^0 $ is given exactly as follows 
$$ u^0  = u_1^0 + i u_2^0 ,$$
where 
\begin{eqnarray*}
u_1^0 & =& T^{ - \frac{1}{p -1}} \left\{  \left( p-1 + \frac{(p-1)^2|x|^2}{ 4 p T |\ln T|}   \right)^{-\frac{1}{p-1}}     + \frac{n \kappa }{2 p |\ln T|} \right.\\
& +& \left. \frac{A}{|\ln T|^2} \left(  d_{1,0} + d_{1,1} \cdot y  \right) \chi_0 \left(\frac{2x}{K \sqrt{T \ln T }} \right) \right\},\\
u_2^0 & =&  T^{ - \frac{1}{p -1}} \left\{ \frac{|x|^2}{ T |\ln T|^2} \left( p-1 + \frac{(p-1)^2|x|^2}{ 4 p T |\ln T|}   \right)^{-\frac{p}{p-1}}     - \frac{2 n \kappa }{(p-1) |\ln T|^2} \right.\\
& +& \left. \left[  \frac{A^2}{|\ln T|^{p_1 +2}} \left(  d_{1,0} + d_{1,1} \cdot y  \right) \chi_0  + \frac{A^5 \ln( |\ln(T)|)}{ |\ln T|^{p_1 + 2}} \left( \frac{1}{2} y^{\mathcal {T} }\cdot  d_{2,2} \cdot 
y  - \text{Tr} (d_{2,2})  \right)  \right]\chi_0\left(\frac{2x}{K \sqrt{T \ln T }} \right) \right\}.
\end{eqnarray*}
with $\kappa = (p-1)^{- \frac{1}{p-1}}$, $K, A$ are positive constants fixed   large enough, $d^{(1)} = (d_{1,0}, d_{1,1}), d^{(2)} = (d_{2,0}, d_{2,1}, d_{2,2}) $   are parametes we fine tune in our proof, and  $\chi_0 \in C^{\infty}_{0}[0,+\infty), \|\chi_0\|_{L^{\infty}} \leq 1,  \text{ supp } \chi_0 \subset [0,2]$. 
\end{remark}
\begin{remark}
We see  below in \eqref{equation-satisfied-u_1-u_2}   that  the equation satisfied  by  of $u_2$  is almost 'linear' in $u_2$. Accordingly,   we may change a little our proof to construct   a solution $u_{c_0} (t)  = u_{1,c_0} + i u_{2,c_0}$ with $t \in [0,T), c_0 \neq 0$, which  blows up in finite time $T$  only at the  origin such that  \eqref{esttima-theorem-profile-complex} and  \eqref{asymp-u-start-near-0-profile-complex}     hold    and also the following 
\begin{equation}
\left\| (T- t)^{\frac{1}{p-1}} |\ln (T-t)|  u_{2,c_0}(x,t) - c_0g_{0}\left( \frac{x}{\sqrt{(T -t) |\ln(T -t)|}}  \right)  \right\|_{L^{\infty}}  \leq \frac{C}{ |\ln (T -t)|^{\frac{p_1}{2}}},
\end{equation}
and 
\begin{equation}
 u^*_2(x)  \sim \frac{2 p c_0}{(p-1)^2}  \left[ \frac{ (p-1)^2|x|^2}{ 8p |\ln|x||}\right]^{-\frac{1}{p-1}}\frac{1}{ |\ln|x||} , \text{ as } x \to 0,
\end{equation}
\end{remark}
\begin{remark} We  deduce from $(ii)$ that $u$  blows up  only at $0$. In particular,  note both the $u_1$ and $u_2$ blow  up.  However, the blowup speed of $u_2$ is   softer  than $u_1$ because of  the quantity  $ \frac{1}{|\ln |x||}$.   
\end{remark}
\begin{remark}
 Nouaili and Zaag constructed a blowup solution of
  \eqref{equation-problem-p=2} with  a less explicit behavior 
  for the imaginary part (see \eqref{asymptotic-imiginary-nouaili-zaag}). Here, we  do better and we obtain the 
  profile the the imaginary part in \eqref{estima-the-imginary-part}  and we also describe the asymptotics of the solution in the neighborhood of the blowup point in \eqref{asymp-u-start-near-0-profile-complex-imaginary-part}. In fact, this refined behavior 
comes from a more involved formal approach (see Section \ref{section-approach-formal} below), and more parameters  to be 
fine tuned in initial data (see Definition \ref{initial-data-profile-complex} where we need more parameters than in Nouaili and Zaag \cite{NZCPDE2015}, namely $d_2 \in \mathbb{R}^{\frac{n(n+1)}{2}}$).   Note also that  our profile estimates  in  \eqref{esttima-theorem-profile-complex} and \eqref{estima-the-imginary-part} are  better than the estimates  \eqref{profile-real-hara} and  \eqref{profile-imaginary-hara}   by Harada ($m = 2$), in the sense that we have a  uniform estimate for whole space $\mathbb{R}^n$, and not just for all $|y| \leq s^{1 + \sigma}$ for some $\sigma >0$. Another point: our result hold in $n$ space dimensions, unlike the work of  Harada in \cite{HaradaJFA2016}, which holds only in one space dimension. 
\end{remark}
\begin{remark}
As in the case $p=2$ treated  by  Nouaili and Zaag \cite{NZCPDE2015},  we suspect this  behavior  in Theorem  \ref{Theorem-profile-complex}  to be unstable. This is due to the fact that the number of parameters in the initial data  we  consider below  in Definition \ref{initial-data-profile-complex} is higher than the dimension of
the blowup parameters which is $n+1$ ($n$ for the blowup points and $1$ for the blowup time).    
\end{remark}

Besides that, we can use the technique of Merle \cite{Mercpam92} 
to construct a solution which blows up at arbitrary given points.
 More precisely, we have  the following Corollary: 
\begin{corollary}[Blowing up at $k$ distinct points]
For any  given points, $x_1,...,x_k$, there exists a solution of \eqref{equ:problem} which blows up exactly at $x_1,...,x_k$. Moreover, the local behavior at each 
blowup point $x_j$ is also given by  \eqref{esttima-theorem-profile-complex}, \eqref{estima-the-imginary-part}, \eqref{asymp-u-start-near-0-profile-complex}, \eqref{asymp-u-start-near-0-profile-complex-imaginary-part}    by replacing $x$ by  $x_j$ and $L^\infty (\mathcal{R}^n)$  by $L^{\infty} (|x - x_j| \leq \epsilon_0),$  for some    $\epsilon_0 >0$.
\end{corollary}

This paper is organized as follows: 

- In Section   \ref{section-approach-formal}, we adopt  a formal approach to  show how the profiles we have in Theorem \ref{Theorem-profile-complex} appear naturally.

- In Section  \ref{section-exisence solu}, we give the rigorous  proof for Theorem \ref{Theorem-profile-complex}, assuming some technical estimates.

- In Section  \ref{the proof of proposion-reduction-finite-dimensional}, we  prove the techical estimates assumed in Section  \ref{section-exisence solu}.

\medskip
\textbf{Acknowledgement:}  I   would like to send a huge thank    to  Professor Hatem ZAAG,  my PhD advisor at Paris 13.    He led  my  first steps of the study.   Not only did he introduced me to the subject, he also gave me valuable indications on the reductions of a mathematics paper. I have no anymore words to describe  this wondeful. Beside that, I also thank  my family who encouraged me  in my mathematical stidies.
\section{Derivation of the profile (formal approach)}\label{section-approach-formal}
In this section, we aim at giveing  a   formal  approach to our problem which helps us   to explain how we derive  the profile of solution of \eqref{equ:problem} given  in Theorem \eqref{Theorem-profile-complex}, as well the   asymptotics of the solution.
\subsection{Modeling  the problem}\label{subsection-pro-L}
In this part, we will give definitions and special symbols important for    our work and  explain how the  functions  $f_0, g_0$ arise as blowup profiles for equation \eqref{equ:problem} as stated in \eqref{esttima-theorem-profile-complex} and \eqref{estima-the-imginary-part}.      Our aim in this section is to give solid (though formal) hints for the existence of a solution $u(t) = u_1 (t)  + i u_2(t)$ to equation  \eqref{equ:problem}  such that 
\begin{equation}\label{lim-u-t-=+-infty}
\lim_{t \to T} \|u(t) \|_{L^\infty} = +\infty,
\end{equation} 
and  $u$ obeys the profiles in \eqref{esttima-theorem-profile-complex} and  \eqref{estima-the-imginary-part}, for some $T > 0$. By using equation  \eqref{equ:problem}, we deduce that $u_1,u_2$ solve: 
\begin{equation}\label{equation-satisfied-u_1-u_2}
\left\{   \begin{array}{rcl}
\partial_t u_1 &=& \Delta u_1 + F_1(u_1,u_2),\\
\partial_t u_2 &=& \Delta u_2 + F_2(u_1, u_2). 
\end{array} \right.
\end{equation}
where 
\begin{equation}\label{defi-mathbb-A-1-2}
\left\{   \begin{array}{rcl}
F_1(u_1,u_2) &=& \text{ Re} \left[( u_1 + i u_2)^p\right]  = \sum_{j=0}^{\left[ \frac{p}{2} \right]}  C^{2j }_p ( -1)^{j} u_1^{p-2j} u_2^{2j}  ,\\
 F_2(u_1, u_2) &=&  \text{ Im} \left[( u_1 + i u_2)^p\right]  = \sum_{j=0}^{\left[ \frac{p-1}{2} \right]}  C^{2j + 1}_p ( -1)^{j} u_1^{p-2j -1} u_2^{2j +1},  
\end{array} \right.
\end{equation}
with   $\text{ Re}[ z] $  and $\text{ Im} [z]$  being  respectively  the real  and the imaginary part  of $z$ and  $C^m_{p} = \frac{p!}{ m ! (p-m)!} ,$ for all $m \leq p .$  Let us introduce \textit{the similarity-variables}: 
\begin{equation}\label{similarity-variales}
w_1 (y,s)  =  (T-t)^{\frac{1}{p-1}}u_1 (x,t), w_2 (y,s)  =  (T-t)^{\frac{1}{p-1}}u_2 (x,t) ,y = \frac{x}{\sqrt {T- t}} , s = - \ln(T-  t).
\end{equation}
 Thanks to  \eqref{equation-satisfied-u_1-u_2},   we  derive the system satisfied by    $( w_1, w_2),$  for all $y \in \mathbb{R}^n$ and $s \geq  -\ln T $ as follows:
\begin{equation}\label{equation-satisfied-by-w-1-2}
\left\{  \begin{array}{rcl}
\partial_s w_1  & = &   \Delta w_1 - \frac{1}{2} y \cdot \nabla w_1  - \frac{w_1}{p-1}   +  F_1(w_1,w_2), \\
\partial_s w_2 &=& \Delta w_2 - \frac{1}{2} y \cdot \nabla w_2  - \frac{w_2}{p-1} + F_2 (w_1,w_2). 
\end{array} \right.
\end{equation}
Then   note that  studying  the    asymptotics  of  $u_1 + i u_2$ as  $t \to T$ is    equivalent to  studying the  asymptotics of $w_1 + i w_2$  in long time.   We  are first interested in the set of constant solutions  of 
 \eqref{equation-satisfied-by-w-1-2}, denoted by 
 $$\mathcal{S} = \left\{  (0,0)\right\}   \cup  \left\{  \left(\kappa \cos\left(\frac{2 k \pi}{p-1} \right) ,\kappa \sin\left(\frac{2 k \pi}{p-1} \right) \right) \text{ where } \kappa 
 = (p-1)^{-\frac{1}{p-1}}, k  = 0,..., p-1 \right\}.$$ 
 With the transformation  \eqref{similarity-variales}, we slightly precise  our goal in \eqref{lim-u-t-=+-infty} by requiring   in addition that
 $$ (w_1, w_2) \to (\kappa, 0)  \text{  as }  s  \to +\infty.$$
Introducing  $w_1 = \kappa + \bar w_1, $ our goal because to get 
$$ (\bar w_1, w_2) \to (0,0) \text{ as } s \to + \infty. $$
From \eqref{equation-satisfied-by-w-1-2}, we deduce  that $\bar w_1, w_2$  satisfy the following system

\begin{equation}\label{system-bar R-varphi}
\left\{   \begin{array}{rcl}
\partial_s \bar w_1 &=& \mathcal{L} \bar w_1  + \bar B_1(\bar w_1, w_2),\\
\partial_s w_2 &=& \mathcal{L} w_2  + \bar B_2 (\bar w_1, w_2).
\end{array} \right.
\end{equation}
where
\begin{eqnarray}
\mathcal{L} &=& \Delta - \frac{1}{2} y \cdot \nabla + Id,\label{define-operator-L}\\
\bar B_1 ( \bar w_1, w_2 )  & = & F_1(\kappa + \bar w_1, w_2)- \kappa^p - \frac{p}{p-1} \bar w_1 , \label{defini-bar-B-1}\\ 
\bar B_2 (\bar w_1, w_2) &=& F_2(\kappa + \bar w_1, w_2) - \frac{p}{p-1} w_2.\label{defini-bar-B-2}
\end{eqnarray}
It is  important to study  the linear operator $\mathcal{L}$  and the asymptotics of $\bar B_1, \bar B_2$ as  $(\bar w_1, w_2) \to (0,0)$ which will appear as quadratic. 

$\bullet $ \textit{ The properties of $\mathcal{L}$:}

We observe that the   operator $\mathcal{L}$  plays an important 
role in our analysis. It is not really difficult to find an analysis space such that $\mathcal{L}$ is
 self-adjoint. Indeed,  $\mathcal{L} $  is self-adjoint in $L^2_\rho(\mathbb{R}^n)$, where $L^2_\rho$ is the weighted space associated with the weight $\rho$ defined by
\begin{equation}\label{def-rho-rho-j}
\rho(y)   =\frac{e^{- \frac{|y|^2}{4}}}{(4 \pi)^{\frac{n}{2}}}  = \prod_{j=1}^n \rho_j (y_j), \text{ with } \rho_j(y_j) = \frac{e^{- \frac{|y_j|^2}{4}}}{(4 \pi)^{\frac{1}{2}}},
\end{equation}
and  the spectrum set of $\mathcal{L}$
$$\textup{spec}(\mathcal{L}) = \displaystyle \left\{1 - \frac m2, m \in \mathbb{N}\right\}.$$
Moreover,  we can find  eigenfunctions which  correspond  to each  eigenvalue $ 1 - \frac{m}{2}, m \in \mathbb{N}$: 
\begin{itemize}
\item[-] The  one space dimensional case:  the eigenfunction corresponding to
 the eigenvalue $1 - \frac m2$   is  $h_m$, the rescaled Hermite polynomial given in   \eqref{Hermite}. In particular, we have the following orthogonality property:
 $$\int_{\mathbb{R}} h_i h_j  \rho dy  =  i! 2^i \delta_{i,j}, \quad \forall (i,j) \in \mathbb{N}^2. $$
\item[-] The higher dimensional case:  $n \geq 2$, the eigenspace  $\mathcal{E}_{m}$, corresponding
 to the eigenvalue $1 - \frac {m}{2}$  is defined as follows:
\begin{equation}\label{eigenspace}
 \mathcal{E}_m   = \left\{ h_{\beta} = h_{\beta_1} \cdots h_{\beta_n}, \text{ for all } \beta \in \mathbb{N}^n, |\beta|  = m , |\beta| = \beta_1 + \cdots +\beta_n  \right\}.
\end{equation}
\end{itemize}
As a matter of fact, so we can represent  an arbitrary function $r \in  L^{2}_\rho$ as follows
\begin{eqnarray*}
r  &=& \displaystyle \sum_{\beta, \beta \in \mathbb{N}^n} r_\beta h_\beta  (y),\\
\end{eqnarray*}
where: $r_\beta$ is the projection of $r$ on $h_\beta $ for any $ \beta \in \mathbb{R}^n$ which is defined as follows:
\begin{equation}\label{projector-h-beta}
r_\beta =  \mathbb{P}_\beta (r)  =  \int r k_\beta \rho dy, \forall  \beta \in \mathbb{N}^n, 
 \end{equation}
with 
\begin{equation}\label{note-k-beta-hermite}
 k_\beta (y) = \frac{ h_\beta}{\|h_\beta\|^2_{L^2_\rho}},
\end{equation}
$\bullet $ \textit{ The  asymptotic of $\bar B_1(\bar w_1, w_2), \bar B_2(\bar w_1, w_2)$:}
The following asymptotics   hold: 
 \begin{eqnarray} 
\bar B_1 (\bar w_1, w_2) &=& \frac{p}{2 \kappa} \bar w_1^2  + O (|\bar w_1|^3 + |w_2|^2),\label{asymptotic-bar-B-1-1}\\
\bar B_2( \bar w_1, w_2) &=& \frac{p}{ \kappa} \bar w_1 w_2  + O \left( |\bar w_1|^2 |w_2| \right) + O  \left( |w_2|^3 \right), \label{asymptotic-bar-B-2-1}                                 
\end{eqnarray}
as  $(\bar w_1, w_2) \to (0,0)$   (see Lemma \ref{asymptotic-bar-B-1-2} below).
\subsection{Inner expansion}\label{subsection-inner-expan}
 In this part, we study    the asymptotics  of the solution 
 in $L^2_\rho(\mathbb{R}^n).$  Moreover, for simplicity  we suppose that  $n =1$, and we recall that  we aim at  constructing a  solution of
  \eqref{system-bar R-varphi} such that 
  $( \bar w_1,  w_2)  \to (0,0) $.   Note first that  the spectrum  of $\mathcal{L}$  contains  two   positive eigenvalues $1, \frac{1}{2}$, a   neutral  eigenvalue $0$ and all  the other ones are  strictly negative. So, in the representation of the solution in $L^2_\rho,$  it is reasonable to think that    the part corresponding to the negative spectrum  is easily  controlled.  Imposing  a symmetry condition  on the solution   with  respect  of $y$,   it is reasonable to look for  a  solution  $ \bar w_1, w_2  $ of the form:
\begin{eqnarray*}
\bar w_1 & = & \bar w_{1,0} h_0  + \bar w_{1,2} h_2, \\
w_2 &=& w_{2,0} h_0 + w_{2,2} h_2. 
\end{eqnarray*}
From the assumption that   $ (\bar w_1, w_2)  \to (0,0)$,  we see that    $ \bar w_{1,0}, \bar w_{1,2}, w_{2,0} , w_{2,2} \to 0$ as $s \to + \infty$. 
We  see also  that we can understand the asymptotics of the solution $\bar w_1, w_2$ in $L^2_\rho$  from  the study of the asymptotics of $\bar w_{1,0}, \bar w_{1,2}, w_{2,0}, w_{2,2}.$
 We  now  project  equations   \eqref{system-bar R-varphi} on $h_0$ and
  $h_2.$ Using   the asymptotics of $\bar B_1, \bar  B_2$  in  \eqref{asymptotic-bar-B-1-1} and \eqref{asymptotic-bar-B-2-1},  we get the following ODEs for $ \bar w_{1,0}, \bar w_{1,2}, w_{2,0} , w_{2,2}:$
\begin{eqnarray}
\partial_s \bar w_{1,0} &=& \bar w_{1,0} + \frac{p}{2 \kappa} \left( \bar w_{1,0}^2 + 8 \bar w_{1,2}^2\right)+ O (|\bar w_{1,0}|^3 + |\bar w_{1,2}|^3) + O( |w_{2,0}|^2 + |w_{2,2}|^2),\label{ODe-bar w_1-0} \\
\partial_s \bar w_{1,2} &=& \frac{p}{\kappa} \left( \bar w_{1,0} \bar w_{1,2} + 4 \bar w_{1,2}^2\right) + O (|\bar w_{1,0}|^3 + |\bar w_{1,2}|^3) + O( |w_{2,0}|^2 + |w_{2,2}|^2) ,\label{ODe-bar w_1-2}\\
\partial_s w_{2,0} &=&  w_{2,0} +  \frac{p}{\kappa}\left[\bar w_{1,0} w_{2,0} + 8 \bar w_{1,2} w_{2,2} \right] + O ((|\bar w_{1,0}|^2 + |\bar w_{1,2}|^2)(|w_{2,0}|+ |w_{2,2}|)) \label{ODe- w_2-0}\\
&+& O( |w_{2,0}|^3 + |w_{2,2}|^3) ,\nonumber\\
\partial_s w_{2,2} &=&   \frac{p}{\kappa} \left[  \bar w_{1,0} w_{2,2} + \bar w_{1,2} w_{2,0} + 8 \bar w_{1,2} w_{2,2}\right] + O ((|\bar w_{1,0}|^2 + |\bar w_{1,2}|^2)(|w_{2,0}|+ |w_{2,2}|)) \label{ODe- w_2-2}\\
&+& O( |w_{2,0}|^3 + |w_{2,2}|^3).   \nonumber
\end{eqnarray}
Assuming that 
\begin{equation}\label{asumption-bar-w-1-0-lesthan-bar-w-1-2}
\bar w_{1,0}, w_{2,0}, w_{2,2} \ll \bar w_{1,2}  \text{ as } s \to + \infty,
\end{equation}
we may simplify the ODE system as follows:

$  \bullet $ \textit{The asymptotics of $\bar w_{1,2}$:}

We deduce from  \eqref{ODe-bar w_1-2} and \eqref{asumption-bar-w-1-0-lesthan-bar-w-1-2} that
$$\partial_s \bar w_{1,2}   \sim \frac{4 p}{\kappa} \bar w_{1,2}^2 \text{ as } s \to + \infty,$$
which yields
\begin{equation}
\bar w_{1,2}  = -\frac{\kappa}{4 p s} + o\left(\frac{1 }{s} \right), \text{ as } s \to + \infty.
\end{equation}
Assuming futher  that
\begin{equation}\label{asuming-bar-w-1-w-2-2-leq-s^2}
\bar w_{1,0}, w_{2,0}, w_{2,2} \lesssim \frac{1}{s^2}, 
\end{equation}
 we see that 
\begin{equation}\label{asymp-bar-w-1-2}
\bar w_{1,2}  = -\frac{\kappa}{4 p s} + O\left(\frac{\ln s }{s^2} \right), \text{ as } s \to + \infty.
\end{equation}

$\bullet $ \textit{The asymptotics of $\bar w_{1,0}:$}
By using \eqref{ODe-bar w_1-0}, \eqref{asumption-bar-w-1-0-lesthan-bar-w-1-2} and the asymptotics of $\bar w_{1,2}$ in \eqref{asymp-bar-w-1-2},  we see that 
\begin{equation}\label{asymp-bar-w-1-0}
\bar w_{1,0} = O \left( \frac{1}{s^2}\right) \text{ as } s  \to + \infty.
\end{equation}

$\bullet $ \textit{The asymptotics of $w_{2,0}$ and $ w_{2,2}$:}
Bisides that, we derive  from \eqref{ODe- w_2-0}, \eqref{ODe- w_2-2} and  \eqref{asuming-bar-w-1-w-2-2-leq-s^2} that
\begin{eqnarray}
\partial_s w_{2,2} &= & \left( -\frac{2}{s}  + O \left( \frac{\ln s}{s^2}\right) \right)  w_{2,2} + o \left( \frac{1}{s^3}\right),\label{ODE-w-2-2-o-1-s-3}\\
\partial_s w_{2,0} & = & w_{2,0}  + O \left(  \frac{1}{s^3}\right),\nonumber
\end{eqnarray}
which  yields
\begin{eqnarray}
w_{2,2} &=& o \left( \frac{\ln s}{s^2}\right),\nonumber\\
w_{2,0} &=& O \left( \frac{1}{s^3} \right),\label{asymptotic-w-2-0}
\end{eqnarray}
as $s \to + \infty$. This also  yields a  new ODE for  $ w_{2,2}:$
$$\partial_s w_{2,2} =  - \frac{2}{s} w_{2,2} + o\left(  \frac{\ln^2 s}{s^4} \right),$$
which implies 
$$ w_{2,2}  = O\left(\frac{1}{s^2} \right).$$
  Using again \eqref{ODE-w-2-2-o-1-s-3}, we    derive a new ODE for $w_{2,2}$
$$\partial_s w_{2,2} =  - \frac{2}{s} w_{2,2} + O\left(  \frac{\ln s}{s^4} \right),$$
which yields 
\begin{eqnarray}
w_{2,2} &=&     \frac{\tilde c_0}{s^2} + O\left(  \frac{\ln s}{s^3}\right), \text{  for some  } \tilde c_0 \in \mathbb{R}^* \label{asymptotic-w-2-2}.
\end{eqnarray}
 Noting that    our finding \eqref{asymp-bar-w-1-2}, \eqref{asymp-bar-w-1-0}, \eqref{asymptotic-w-2-0} and \eqref{asymptotic-w-2-2} are consistent with our hypotheses in \eqref{asumption-bar-w-1-0-lesthan-bar-w-1-2} and \eqref{asuming-bar-w-1-w-2-2-leq-s^2}, we get   the asymptotics of the solution $w_1$ and $w_2$ as follows:
\begin{eqnarray}
w_1&=& \kappa - \frac{\kappa}{4p s } ( y^2 - 2) + O\left(\frac{1}{s^2} \right), \label{asymptotic-w-1}\\
w_2  &=& \frac{\tilde c_0}{s^2} (y^2 - 2)  + O \left( \frac{\ln s}{s^3}\right), \label{asymptotic-w-2},
\end{eqnarray}
in   $L^2_\rho(\mathbb{R})$  for some  $\tilde c_0$ in $\mathbb{R}^*$.
Using parabolic  regularity, we note  that the asymptotics \eqref{asymptotic-w-1}, \eqref{asymptotic-w-2}   also hold for all $|y| \leq K,$ where $K$ is an arbitrary positive constant.
\subsection{Outer expansion}
 As Subsection  \ref{subsection-inner-expan} above,  we   assume that  $n=1$.  We   see  that asymptotics \eqref{asymptotic-w-1} and  \eqref{asymptotic-w-2}
 can not give us a shape, since they  hold uniformly  on compact sets, and not in larger sets.  Fortunately,  we observe from  \eqref{asymptotic-w-1} and \eqref{asymptotic-w-2} that the profile may be based on the following variable:    
\begin{equation}\label{the-variable-z-profile}
z  = \frac{y}{\sqrt s}.
\end{equation}
This  motivates us  to look for  solutions  of the form:
\begin{eqnarray*}
w_1(y,s) &=& \sum_{j=0}^{\infty} \frac{R_{1,j} (z)}{s^j}, \\
w_2 (y,s)  &=& \sum_{j=1}^{\infty} \frac{R_{2,j}(z)}{s^j}.
\end{eqnarray*}
Using  system  \eqref{equation-satisfied-by-w-1-2} and  gathering terms  of order $\frac{1}{s^j}$ for $j=0,...,2$,   we obtain 
\begin{eqnarray}
0 &=&  - \frac{1}{2} R_{1,0}' (z) \cdot z - \frac{R_{1,0} (z)}{p-1} + R_{1,0}^p (z), \label{equa-R-1-0} \\
0 &=&  - \frac{1}{2} z R_{1,1}' - \frac{R_{1,1}}{p-1} + p R_{1,0}^{p-1}R_{1,1} + R_{1,0}'' 
 + \frac{z R_{1,0}'}{2}, \label{equa-R-1-1}  \\
0 &=&  - \frac{1}{2}  R_{2,1}' (z) \cdot z - \frac{R_{2,1}}{p-1} + p R_{1,0}^{p-1} R_{2,1}, \label{equa-R-2-1}\\
0 &=&  - \frac{1}{2}  R_{2,2}'(z). z - \frac{R_{2,2}}{p-1} + p R^{p-1}_{1,0} R_{2,2} + R''_{2,1} + R_{2,1} + \frac{1}{2} R_{2,1}' \cdot z + p (p-1) R^{p-2}_{1,0} R_{1,1} R_{2,1}. \label{equa-R-2-2}
\end{eqnarray}
We now solve  the above equations:

$  \bullet $ \textit{ The solution $R_{1,0}$:} It is easy to solve  \eqref{equa-R-1-0}
\begin{equation} \label{solu-R-0}
R_{1,0} (z) = (p-1  + b z^2)^{- \frac{1}{p-1}},
\end{equation}
where $b $ is an  unknown constant that  will be selected accordingly to our purpose.

$  \bullet $ \textit{ The solution $R_{1,1}$:}  We rewrite   \eqref{equa-R-1-1} under  the following form:
$$ \frac{1}{2} z. R_{1,1}'(z) = \left( \frac{(p-1)^2 - bz^2}{(p-1)(p-1 + b z^2)} \right)R_{1,1}  + F_{1,1}(z),$$
where
\begin{eqnarray*}
F_{1,1} (z)  &=&  - \frac{2 b }{p-1} (p-1 + bz^2)^{- \frac{p}{p-1}} + \frac{4 p b^2 z^2}{(p-1)^2} (p-1 + bz^2)^{- \frac{(2p-1)}{p-1}} \\
&  - &  \frac{b z^2}{p-1} (p-1 + b z^2)^{- \frac{p}{p-1}}.
\end{eqnarray*}
Thanks to the variation of constant method, we see that
\begin{eqnarray}\label{variation-constant-R-1-1}
R_{1,1} = H^{-1} (z) \left(  \int  \frac{2 }{z}H (z) F_{1,1} (z) dz  + C_1 \right), 
\end{eqnarray}
where
$$ H (z)   = \frac{ (p-1 +  bz^2)^{\frac{p}{p-1}}}{ z^2}.$$
Besides that, we have:
\begin{eqnarray*}
\frac{2 H}{z} F_{1,1} &=& - \frac{4 b }{ (p-1) z ^3 } + \frac{8p b^2}{(p-1)^2} \left(  \frac{1}{ z (p-1  + b z^2)} \right)  -  \frac{2b}{ (p-1) z}\\
&=& - \frac{4 b }{ (p-1) z^3} + \frac{1}{ z} \left( -\frac{2b}{p-1} + \frac{8p b^2}{(p-1)^3}  \right) \\
&+& (p-1 + b z^2)^{-1} \left( - \frac{8 p b^3 z}{(p-1)^3} \right)\\
\end{eqnarray*}
We can see that  if the coefficient of $\frac{1}{z}$ is non zero, then we will have a $\ln z$ term in the solution $R_{1,1}$ and this term would  not be analytic, creating a singularity in the solution.   In order to avoid this singularity, we impose that 
\begin{equation*}\label{condition-regularity-R-1}
-\frac{2b}{p-1} + \frac{8p b^2}{(p-1)^3} = 0.
\end{equation*}
which yields
\begin{equation}\label{condition-b}
b= \frac{(p-1)^2}{4 p }.
\end{equation}
Besides that,  for simplicity,  we assume that  $C_1 = 0.$ Using   \eqref{variation-constant-R-1-1},   we  see that 
\begin{eqnarray}
R_{1,1} &=& \frac{(p-1)}{ 2 p} (p-1 + bz ^2)^{- \frac{p}{p-1}} -  \frac{p-1}{4 p} z^2 \ln (p-1 + b z^2) (p-1 + b z^2)^{- \frac{p}{p-1}}. \label{solu-R-1-1}
\end{eqnarray}

$  \bullet $ \textit{ The solution $R_{2,1}$:}  It is easy to solve \eqref{equa-R-2-1} as follows:
\begin{equation}\label{solu-varphi_1}
R_{2,1} (z) = \frac{z^2}{(p-1 + bz^2)^{\frac{p}{p-1}}}.
\end{equation}

$  \bullet $ \textit{ The solution $R_{2,2}$:}  We rewrite \eqref{equa-R-2-2} as follows 
\begin{eqnarray*} 
\frac{1}{2} z \cdot R_{2,2}'(z) &=&    \left( \frac{(p-1)^2 - bz^2}{(p-1)(p-1 + b z^2)} \right) R_{2,2} (z) + F_{2,2}(z),
 \end{eqnarray*}
where  
\begin{eqnarray*}
F_{2,2}(z) &=& R''_{2,1} + R_{2,1} + \frac{1}{2} R_{2,1}' \cdot z + p (p-1) R^{p-2}_{1,0} R_{1,1} R_{2,1} \\
&=&  2 (p-1 + b z^2)^{- \frac{p}{p-1}} \\
&-& \frac{10 p b z^2}{p-1} (p-1+ b z^2)^{- \frac{2p-1}{p-1}} + 2 z^2 (p-1 + b z^2)^{-\frac{p}{p-1}} + \frac{(p-1)^2}{2} z^2 (p-1 + b z^2)^{- \frac{3p-2}{p-1}} \\
&+& \frac{4p (2p-1)b^2z^4}{(p-1)^2} (p-1 + b z^2)^{-\frac{3p-2}{p-1}} -\frac{p b z^4}{p-1} (p-1 + b z^2)^{- \frac{2p-1}{p-1}} \\
&- &\frac{(p-1)^2}{4} z^4 \ln (p-1 + b z^2) (p-1+ b z^2)^{- \frac{3p-2}{p-1}}.
\end{eqnarray*}
By using the variation of  constant method, we have
\begin{eqnarray}
R_{2,2} (z) = \frac{z^2}{(p-1+ b z^2)^{-\frac{p}{p-1}}} \left( \int   \frac{2 (p-1+ b z^2)^{-\frac{p}{p-1}}}{z^3}  F_{2,2}(z) dz + C_2 \right),\label{variation-constant-R-2-2}
\end{eqnarray}
where 
\begin{eqnarray*}
\frac{2 (p-1+ b z^2)^{-\frac{p}{p-1}}}{z^3}  F_{2,2}(z) &=& \frac{4}{z^3} + \left[  5 - \frac{20 p b}{(p-1)^2}\right] \frac{1}{z} + \frac{z}{p-1+ b z^2} \left[ \frac{20p b}{(p-1)^2} - b - \frac{2 p b}{p-1} \right] \\
&+& \left[ \frac{8 p (2p-1) b^2}{(p-1)^2}  -  (p-1) p\right] \frac{z}{(p-1+ b z^2)^2} \\
&-& \frac{(p-1)^2}{2} z \ln (p-1 + b z^2) (p-1 + b z^2)^{-2}. 
\end{eqnarray*}
We observe that 
$$ 5 - \frac{20 p b}{(p-1)^2}  = 0, \text{ because  }  b = \frac{(p-1)^2}{4 p }.$$
So, from  \eqref{variation-constant-R-2-2} and  assuming  that $C_2 = 0,$  we have
\begin{equation}\label{solu-R-2-2}
R_{2,2} (z) = - 2 (p-1 + b z^2) ^{- \frac{p}{p-1}} + H_{2,2} (z),
\end{equation}
where
\begin{eqnarray*}
H_{2,2} (z)  &=& C_{2,1} (p) z^2 (p-1 + b z^2)^{-\frac{2p-1}{p-1}} + C_{2,3} (p) z^2 \ln (p-1 + b z^2) (p-1 + b z^2)^{-\frac{p}{p-1}} \\
&+& C_{2,3} (p) z^2 \ln (p-1 + b z^2) (p-1 + b z^2)^{-\frac{2p-1}{p-1}}.
\end{eqnarray*}
\subsection{Matching  asymptotics}
  Since the outer expansion has to match the inner expansion, this will  fiw several constant, giving us     the   following   profiles  for  $w_1$ and $w_2:$
 \begin{equation}\label{equavalent-w-1-2-Phi-1-2}
\left\{   \begin{array}{rcl}
w_1 (y,s) &\sim &  \Phi_1(y,s),\\
w_2 (y,s) &\sim & \Phi_2(y,s),
\end{array} \right.
\end{equation}
 where 
\begin{eqnarray}
\Phi_1(y,s) &=& \left( p-1 + \frac{(p-1)^2}{4 p} \frac{|y|^2}{s} \right)^{-\frac{1}{p-1}} + \frac{n \kappa}{2 p s},\label{defi-Phi-1}\\
\Phi_2 (y,s) &=&\frac{|y|^2}{s^2} \left( p-1 + \frac{(p-1)^2}{4 p} \frac{|y|^2}{s} \right)^{-\frac{p}{p-1}} -  \frac{2n \kappa}{(p-1) s^2},\label{defi-Phi-2}
\end{eqnarray}
for all $(y,s) \in \mathbb{R}^n \times (0,  + \infty)$.
\section{Existence of  a blowup  solution  in Theorem  \ref{Theorem-profile-complex}}\label{section-exisence solu}
In Section \ref{section-approach-formal}, we adopted a formal approach on order to justify how the profiles $f_0, g_0$ arise as blowup profiles for equation \eqref{equ:problem}.  In this section, we give a rigorous  proof to justify the existence  of a solution approaching those profiles.  
\subsection{Formulation of the problem}
 In this section, we aim at formulating  our problem in order  to justify  the  formal  approach which is given in the previous  section. Introducing
 \begin{equation}\label{defini-q-1-2}
 \left\{   \begin{array}{rcl}
 w_1 &=& \Phi_1  + q_1, \\
w_2 &=& \Phi_2  + q_2,
\end{array}  \right.
 \end{equation}
where      $\Phi_1, \Phi_2$ are defined in \eqref{defi-Phi-1} and \eqref{defi-Phi-2} respectively,     then using  \eqref{equation-satisfied-by-w-1-2},  we see that $(q_1, q_2) $ satisfy
\begin{equation}\label{equation-satisfied-by-q-1-2}
\partial_s \binom{q_1}{q_2} = \left( \begin{matrix}
\mathcal{L} + V  & 0 \\
  0 & \mathcal{L} + V
\end{matrix} \right) \binom{q_1}{q_2} +   \left( \begin{matrix}
 V_{1,1}  &  V_{1,2} \\
 V_{2,1 }& V_{2,2}
\end{matrix} \right)  \binom{q_1 }{q_2}  + \binom{B_1}{B_2} \binom{q_1}{q_2} + \binom{R_1 (y,s)}{R_{2} (y,s)}
\end{equation}
where linear operator $\mathcal{L}$ is defined   in \eqref{define-operator-L} and:\\

\medskip
\noindent
 - The potential functions $V, V_{1,1}, V_{1,2}, V_{2,1}, V_{2,2} $ are defined  as follows 
\begin{eqnarray}
V(y,s) &=&  p \left( \Phi_1^{p- 1}  - \frac{1}{p-1}\right)\label{defini-potentian-V}, \\
V_{1,1} (y,s) & = &  \sum_{  j=1}^{ \left[\frac{p}{2}\right]}  C_p^{2j} (-1)^j (p-2j) \Phi_1^{p - 2j -1} \Phi_2^{2j} ,\label{defini-V-1-1} \\
V_{1,2} (y,s) & = &  \sum_{  j=0}^{ \left[\frac{p}{2}\right]}  C_p^{2j} (-1)^j .(2j) \Phi_1^{p - 2j} \Phi_2^{2j - 1}, \label{defini-V-1-2} \\
V_{2,1} (y,s) &  = &     \sum_{  j=0}^{ \left[\frac{p-1}{2}\right]}  C_p^{2j+ 1} (-1)^j (p-2j  -1) \Phi_1^{p - 2j -2} \Phi_2^{2j+ 1}, \label{defini-V-2-1}  \\
V_{2,2} (y,s) & = & \sum_{  j =1 }^{ \left[\frac{p-1}{2}\right]}  C_p^{2j+ 1} (-1)^j (2j + 1) \Phi_1^{p - 2j  -1} \Phi_2^{2j }. \label{defini-V-2-2}
\end{eqnarray}
\medskip
\noindent
 - The quadratic terms $B_1 (q_1, q_2), B_2 (q_1,q_2)$  are defined as follows:
\begin{eqnarray}
B_1 (q_1,q_2) & = &  F_1 \left(  \Phi_1 + q_1, \Phi_2 + q_2 \right) - F_1(\Phi_1, \Phi_2)  -  \sum_{j=0 }^{ \left[\frac{p}{2}\right]}  C_p^{2j} (-1)^j (p-2j) \Phi_1^{p - 2j -1} \Phi_2^{2j}  q_1 \label{defini-quadratic-B-1}\\
&-&  \sum_{  j=0}^{ \left[\frac{p}{2}\right]}  C_p^{2j} (-1)^j .(2j) \Phi_1^{p - 2j} \Phi_2^{2j - 1} q_2,\nonumber\\
B_2(q_1, q_2) & = &   F_2 \left(  \Phi_1 + q_1, \Phi_2 + q_2 \right) - F_2(\Phi_1, \Phi_2) -  \sum_{  j=0}^{ \left[\frac{p-1}{2}\right]}  C_p^{2j+ 1} (-1)^j (p-2j  -1) \Phi_1^{p - 2j -2} \Phi_2^{2j+ 1} q_1  \nonumber\\
&  -  &  \sum_{  j =0 }^{ \left[\frac{p-1}{2}\right]}  C_p^{2j+ 1} (-1)^j (2j + 1) \Phi_1^{p - 2j  -1} \Phi_2^{2j } q_2.\label{defini-term-under-linear-B-2}
\end{eqnarray}

\medskip
\noindent
 -  The rest terms $R_1(y,s), R_2(y,s)$ are defined as follows:
\begin{eqnarray}
R_1 (y,s) &=& \Delta \Phi_1 - \frac{1}{2} y \cdot \nabla \Phi_1 - \frac{\Phi_1}{p-1} + F_1 (\Phi_1, \Phi_2) - \partial_s \Phi_1 , \label{defini-the-rest-term-R-1}\\
R_2 (y,s) &=& \Delta \Phi_2 - \frac{1}{2} y \cdot \nabla \Phi_2  - \frac{\Phi_2}{p-1} + F_2 (\Phi_1, \Phi_2) - \partial_s \Phi_2, \label{defini-the-rest-term-R-2}
\end{eqnarray}
where   $ F_1, F_2$ are defined in  \eqref{defi-mathbb-A-1-2}.

By the  linearization around $\Phi_1, \Phi_2,$  our  problem is reduced to  constructing  a solution $(q_1,q_2)$ of system \eqref{equation-satisfied-by-q-1-2}, satisfying
$$ \|q_1\|_{L^{\infty}(\mathbb{R}^n)} + \|q_2\|_{L^{\infty}(\mathbb{R}^n)} \to 0 \text{ as } s \to +\infty.$$
 Concerning  equation \eqref{equation-satisfied-by-q-1-2}, we recall that we  already know   the properties of  the  linear operator $\mathcal{L}$  (see page \pageref{define-operator-L}).  As  for     potentials $V_{j,k}$  where $  j,k  \in  \{1,2\},$ they  admit the following asymptotics
\begin{eqnarray*}
\sum_{j,k \leq 2}  |V_{j,k} (y,s) | \leq \frac{C}{s}, \forall   y \in \mathbb{R}^n, s\geq 1,
\end{eqnarray*}
(see   Lemma \ref{lemmas-potentials}). Regarding  the terms $B_1,B_2, R_1, R_2$, we see that whenever $|q_1| + |q_2| \leq 2,$ we have 
\begin{eqnarray*} 
|B_1(q_1,q_2)| &\leq & C(q_1^2 + q_2^2),\\
|B_2(q_1,q_2)| &\leq& C  \left( \frac{|q_1|^2}{s} + |q_1 q_2| + |q_2|^2 \right),\\
\|R_1(y,s)\|_{L^\infty(\mathbb{R}^n)} &\leq & \frac{C}{s}, \\
\|R_2 (.,s)\|_{L^\infty(\mathbb{R}^n)}  &\leq &  \frac{C}{s^2},\\
\end{eqnarray*}
(see  Lemmas  \ref{lemma-quadratic-term-B-1-2} and \ref{lemma-rest-term-R-1-2}). In fact, the dynamics of  equation  \eqref{equation-satisfied-by-q-1-2} will mainly depend on the  main  linear operator 
$$ \left(   \begin{matrix}
\mathcal{L} + V  & 0\\
0  & \mathcal{L} + V
\end{matrix}  \right), $$
and  the effects  of the orther terms  will be less important. For that reason, we need to understant  the dynamics of  $\mathcal{L} + V$. Since  the spectral properties  of  $\mathcal{L}$ were already  introduced  in Section \ref{subsection-pro-L}, we will focus here on the effect of $V$. 

$i)$ Effect of $V$ inside the blowup region $\{|y| \leq K\sqrt s\}$ with $K>0$ arbitrary, we have 
$$  V  \to 0  \text{ in  } L^2_\rho(|y| \leq K \sqrt s ) \text{ as } s \to + \infty,$$
which means that the effect of $V$ will be negligeable with  respect of the effect  of $\mathcal{L},$ except perhaps on the null mode of $\mathcal{L}$ (see item $(ii)$  of Proposition  \ref{prop-dynamic-q-1-2-alpha-beta}   below) 

$ii)$ Effect  of  $V$ outside   the blowup region:   for each  $\epsilon > 0,$ there exist $K_{\epsilon} >0$ and $ s_{\epsilon} >0$ such that
$$ \sup_{\frac{y}{\sqrt s} \geq K_{\epsilon}, s \geq s_{\epsilon}} \left| V(y,s)  + \frac{p}{p-1} \right| \leq    \epsilon.$$ 
Since $1$ is the biggest eigenvalue of $\mathcal{L}$, the operator $\mathcal{L}+ V$  behaves as one with with a fully negative spectrum  outside blowup region $\{|y| \geq K_\epsilon\sqrt s\}$, which makes the control of the solution in this region easily.

\medskip
Since the behavior of the potential $V$ inside and outside the blowup region is different,
 we will consider the dynamics of the solution for   $|y| \leq 2K\sqrt s$ and for $|y| \geq K\sqrt s$ separately for some   $K$  to be fixed large.
For  that purpose,     we introduce the following cut-off function
\begin{equation}\label{def-chi}
\chi(y,s)  = \chi_0\left(\frac{|y|}{K \sqrt s} \right),
\end{equation}
where $\chi_0 \in C^{\infty}_{0}[0,+\infty), \|\chi_0\|_{L^{\infty}} \leq 1$ and 
$$
\chi_0(x) = \left\{  \begin{array}{l}
 1 \quad \text{ for } x  \leq 1,\\
 0 \quad   \text{ for }  x  \geq 2,
\end{array} \right.$$
and  $K$ is a positive constant to be fixed large later. Hence,  it is  reason able  to  consider  separately  the solution  in   the blowup region  $\{ |y| \leq 2 K \sqrt s \}$   and  in the  regular region  $\{| y| \geq  K \sqrt s \}$.    More precisely,  let us define the following notation for all functions   $q $ in   $L^\infty$ as follows 
\begin{equation}\label{defini-q-1-1-e}
 q =  q_b  + q_e     \text{ with }  q_b =   \chi  q  \text{ and  }   q_e   = (1 - \chi) q,
\end{equation}
 Note  in particular   that $\text{ supp} (q_b) 	\subset \mathbb{B} ( 0, 2 K \sqrt s)$ and  $  \text{ supp} (q_e) 	     \subset  \mathbb{R}^n \setminus   \mathbb{B} ( 0,  K \sqrt s)$.
Besides that,   we  also expand    $q_b $ in $L^2_\rho$ as follows; according to the spectrum of $\mathcal{L}$  (see Sention \ref{subsection-pro-L} above):
\begin{equation}\label{representation-q-1-L-2-rho}
 q_b (y) =    q_0 + q_1 \cdot y  +  \frac{1}{2} y^{\mathcal{T}}  \cdot  q_2  \cdot y    -  \text{ Tr} \left(  q_2 \right)  + q_- (y) ,
\end{equation}
where  
\begin{eqnarray*}
q_0  &  =  &  \int_{\mathbb{R}^n}  q_b \rho (y) d y, \\
q_1  &=&   \int_{\mathbb{R}^n}  q_b \frac{y}{2}  \rho (y) d y, \\
q_2 & =&    \left (  \int_{\mathbb{R}^n}  q_b  \left(  \frac{1}{4} y_j y_k -  \frac{1}{2} \delta_{j,k} \right)  \rho (y) d y \right)_{1 \leq j,k \leq n},\\
\end{eqnarray*}
and  $\text{ Tr }(q_2)$ is the trace of the matrix $q_2$.
The reader   should  keep  in mind that $q_0, q_1,q_2$ are just  coordinates of $q_b$, not for  $q$.  Note that $q_m$ is  the projection of  $q_b$ as the eigenspace  of $\mathcal{L}$ corresponding to the eigenvalue  $\lambda = 1 - \frac{m}{2}.$ Accordingly, $q_-$ is the projection of $q_b$  on the negative part of the spectrum  of $\mathcal{L}.$ As a consequence of \eqref{defini-q-1-1-e}  and \eqref{representation-q-1-L-2-rho}, we see that every $q \in L^\infty (\mathbb{R}^n)$  can be decomposed into $5$ components as follows:
\begin{equation}\label{decom-5-parts}
q = q_b + q_e  = q_0 + q_1 \cdot y + \frac{1}{2} y^\mathcal{T}  \cdot q_2 \cdot y - \text{Tr} (q_2) + q_- + q_e.
\end{equation} 
\subsection{The shrinking set}
In this part, we  will  construct a shrinking set,    such that  the control  of $(q_1,q_2) \to 0,$  will be a consequence of      the  control  of     $(q_1,q_2)$ in this shrinking set. This is  our definition 
\begin{definition}[The shrinking set]\label{the-rhinking set}
 For all $A \geq  1, p_1 \in (0,1)$ and  $s > 0,$   we introduce the set $V_{p_1,A,} (s) $ denoted for simplicity by $ V_A (s)   $  as the  set of all $(q_1, q_2) \in (L^\infty (\mathbb{R}^n))^2$    satisfying the following conditions: 
\begin{eqnarray*}
|q_{1,0} | \leq \frac{A}{s^2} &\text{ and }&   |q_{2,0}| \leq \frac{A^2}{s^{p_1 + 2}},\\
|q_{1,j} | \leq \frac{A}{s^2} &\text{ and }&   |q_{2,j} | \leq \frac{A^2}{s^{p_1 + 2}}, \forall 1 \leq  j \leq n,\\
|q_{1,j,k} | \leq \frac{A^2 \ln s}{s^2} &\text{ and }&   |q_{2,j,k} | \leq \frac{A^5 \ln s}{s^{p_1 + 2}}, \forall 1 \leq j,k \leq n,\\
\left\| \frac{q_{1,-} }{1 + |y|^3} \right\|_{L^\infty} \leq \frac{A}{s^{2}} &\text{ and }&  \left\| \frac{q_{2,-} }{1 + |y|^{3}} \right\|_{L^\infty} \leq \frac{A^2}{ s^{\frac{p_1 + 5}{2}}},\\
\|q_{1,e} \|_{L^\infty} \leq \frac{A^2}{\sqrt s} & \text{ and } &  \|q_{2,e} \|_{L^\infty} \leq \frac{A^3}{ s^{\frac{p_1 + 2}{2}}},
\end{eqnarray*}
where    $q_1$ and  $q_2$ are decomposed as in \eqref{decom-5-parts} .
\end{definition}
In the following Lemma, we show that belonging to $V_A(s)$  implies the convergence to $0$.  In fact, we have a more precise statement in the following:

\begin{lemma}\label{lemma-estiam-q-1-2-in V-A}
For all $A \geq 1, s \geq 1,$  if we have   $(q_1, q_2) \in V_A (s)$, then the following  estimates  hold:
\begin{itemize}
\item[$(i)$] $\|q_1\|_{L^\infty (\mathbb{R}^n) } \leq \frac{C A^2}{ \sqrt s} \text{ and }  \|q_2\|_{L^\infty(\mathbb{R}^n)} \leq \frac{CA^3}{s^{\frac{p_1 + 2 }{2}}}.$
\item[$(ii)$] 
$$ |q_{1,b} (y) | \leq \frac{CA^2 \ln s}{s^2} (1 + |y|^3), \quad |q_{1,e} (y)| \leq \frac{C A^2}{s^2} (1 + |y|^3) \text{ and } 	 |q_1| \leq  \frac{C A^2 \ln s}{ s^2} (1 + |y|^3),$$
and
$$ |q_{2,b} (y) | \leq \frac{CA^3 }{s^{\frac{p_1 +5}{2}}} (1 + |y|^3), \quad |q_{2,e} (y)| \leq \frac{C A^3}{s^{\frac{p_1 + 5}{2}}} (1 + |y|^3) \text{ and } 	 |q_2| \leq  \frac{C A^3 \ln s}{ s^{\frac{p_1 + 5}{2}}} (1 + |y|^3).$$
\item[$(iii)$]  For all $y \in \mathbb{R}^n$ we have
$$   |q_1| \leq  C  \left[   \frac{A}{s^2}(  1 + |y| )  + \frac{A^2 \ln s}{s^2} (1 + |y|^2)  +  \frac{A^2}{s^2} (1  + |y|^3) \right],$$
and 
$$ |q_2| \leq C \left[   \frac{A^2}{s^{p_1 + 2}} (1 + |y|)  +  \frac{A^5 \ln s}{s^{p_1 + 2}} (1  + |y|^2)    +  \frac{A^3}{ s^{\frac{p_1 + 5}{2}}} (1  + |y|^3)      \right] .$$
\end{itemize}

where  $C$ will henceforth  be an universal constant in our proof which depends only on $K $. 
\end{lemma}
\begin{proof}
We only prove the estimate for $q_2$ since the estimates for $q_1$ follow similarly and has already been proved in previous papers (see for intance  Proposition 4.7  in  \cite{TZpre15}).  We now take  $A \geq  1, s \geq 1$ and $(q_1, q_2)  \in V_A (s)$ and  $y \in \mathbb{R}^n$.  We also  recall  from \eqref{decom-5-parts} that
$$q_2   = q_{2,b}  + q_{2,e},$$
where  $\text{ supp} (q_{2,b}) 	\subset \mathbb{B} ( 0, 2 K \sqrt s)$ and  $  \text{ supp} (q_{2,e}) 	     \subset  \mathbb{R}^n \setminus   \mathbb{B} ( 0,  K \sqrt s)$.

$(i)$    From   \eqref{representation-q-1-L-2-rho},   we have
$$  q_b = q_{2,0} + q_{2,1} \cdot y + \frac{1}{2} y^\mathcal{T}  \cdot q_{2,2} \cdot y - \text{Tr} (q_{2,2}) + q_{2,-} .$$
Therefore,
\begin{eqnarray}
|q_{2,b} (y)| & \leq & |q_{2,0}|  + |q_{2,1}|  |y| + \max_{j,k \leq  n} |  q_{2,j,k} | (1  + |y|^2)   + \left\| \frac{q_{2,-}}{1 + |y|^3}\right\|_{L^\infty (\mathbb{R}^n)} (1 + |y|^3).\label{modul-q-2-b}
\end{eqnarray}
Then, recalling that  $\text{supp} (q_{2,b})  \subset  \mathbb{B} ( 0, 2 K \sqrt s), $ using Definition \ref{the-rhinking set}, we see that  
$$ |q_{2,b} (y)| \leq  \frac{CA^3}{ s^{\frac{p_1 + 2}{2}}}.$$
 Since we also have
$$ | q_{2,e} | \leq \frac{A^3}{s^{\frac{p_1 + 2}{2}}}.$$
We end-up  with
$$ \|  q_{2}  \|_{L^\infty}  \leq \| q_{2,b}\|_{L^\infty}  + \| q_{2,e}\|_{L^\infty} \leq \frac{C A^3}{  s^{\frac{p_1 + 1}{2}}} .$$

$(ii)$ Using \eqref{modul-q-2-b}  and Definition \ref{the-rhinking set}, we see that
\begin{equation}\label{estima-q-2-b-s-p-1-5}
| q_{2,b} (y) |  \leq  \frac{C A^3  }{ s^{\frac{p_1 + 5 }{2}}}  ( 1 + | y|^3).
\end{equation}
We claim that   $q_{2,e}$ satisfies a similar estimate:
\begin{equation}\label{estima-q-2-e-s-p-1-5}
| q_{2,e} (y)  |  \leq  \frac{C A^3  }{ s^{\frac{p_1 + 5 }{2}}}  ( 1 + | y|^3).
\end{equation}
Indeed, since  $\text{ supp} (q_{2,e}) 	     \subset  \mathbb{R}^n \setminus   \mathbb{B} ( 0,  K \sqrt s),$ we may assume that 
$$   \frac{|y| }{K  \sqrt s}   \geq 1,$$
hence, from Definition \ref{the-rhinking set}, we write 
$$ |q_{2,e}(y)| \leq \frac{A^3}{s^{\frac{p_1  + 2}{2}}} .1 \leq  \frac{A^3}{s^{\frac{p_1 + 2}{2}}}  \frac{|y|^3}{ K^3 s^\frac{3}{2}} \leq  \frac{C A^3 }{s^\frac{p_1  + 5}{2}}   ( 1 + |y|^3),$$ 
and     \eqref{estima-q-2-e-s-p-1-5} follows. Using \eqref{estima-q-2-b-s-p-1-5} and   \eqref{estima-q-2-e-s-p-1-5}, we see that 
$$ |q_2|  \leq | q_{2,b}  |   + |q_{2,e}| \leq \frac{C A^3}{s^\frac{p_1 + 5}{2} } (1 + |y|^3).$$
$(iii)$    It is leaved to reader,  since  this is a direct consequence  of Definition \eqref{the-rhinking set} and  the decomposition  \eqref{decom-5-parts}.
 \end{proof}

\subsection{Initial data}
 Here we    suggest   a class of initial data,  depending on some parameters to be fine-tuned  in order to get a  good solution for our problem.   This is  initial data:
\begin{definition}[The initial data]\label{initial-data-profile-complex} For each
  $A \geq 1, s_0 \geq 1, d_1= (d_{1,0}, d_{1,1}) \in \mathbb{R} \times \mathbb{R}^n,  d_2 = (d_{2,0},d_{2,1}, d_{2,2}) \in \mathbb{R}  \times \mathbb{R}^{  n}  \times \mathbb{R}^{\frac{n(n+1)}{2}}$, we introduce
\begin{eqnarray*}
\phi_{1,A,d_1,s_0} (y) &=& \frac{A}{s_0^2} \left(  d_{1,0} + d_{1,1} \cdot  y \right) \chi (2 y, s_0),\\
  \phi_{2,A, d_2,s_0} (y) &= &\left( \frac{A^2}{s_0^{p_1+2}} \left(   d_{2,0} +  d_{2,1}  \cdot  y \right) + \frac{A^5 \ln s_0}{s^{p_1+2}_0}\left(  y^{\mathcal{T}} \cdot d_{2,2}\cdot y - 2\textup{ Tr}(d_{2,2}) \right) \right)\chi (2 y, s_0).
\end{eqnarray*}
\end{definition}
\textbf{Remark:} Note  that $d_{1,0}$ and $d_{2,0}$ are scalars, $d_{1,1}$ and $d_{2,1}$ are vectors, $d_{2,2}$ is a square  matrix of order $n$.  For simplicity, we may drop down the parameters expect $s_0$ and  write $\phi_1 (y,s_0)$ and $\phi_2 (y,s_0)$.

\medskip
We next claim that we can  find a    domain for   $(d_{1},d_{2})$  so that  initial data   belongs  to    $V_A(s_0):$ 
\begin{lemma}[Control  of initial data in $V_A(s_0)$]\label{lemma-control-initial-data}
There exists $A_1 \geq 1$ such that  for all $ A \geq A_1$, there exists $s_1(A) \geq 1$ such that for all $s_0 \geq s_1(A),$ if   $(q_1,  q_2) (s_0)= \left( \phi_{1},  \phi_{ 2}\right)(s_0) $ where $(\phi_{1},  \phi_{ 2})(s_0)$ are defined in Definition \ref{initial-data-profile-complex},  then,   the following properties hold:
\begin{itemize}
\item[$i)$] There exists a set $\mathcal{D}_{A, s_0}   \in \left[ -2, 2\right]^{ \frac{n^2 + 5n + 4}{2}} $ such that   the mapping
\begin{eqnarray*}
 \Psi_1:  \mathbb{R}^{\frac{n^2 + 5n + 4}{2}} &\to & \mathbb{R}^{\frac{n^2 + 5n + 4}{2}}  \\
(d_1,  d_2) & \mapsto & ( q_{1,0},  (q_{1,j} )_{j \leq n } , q_{2,0},  (q_{2,j})_{j\leq n}, (q_{2,j,k} )_{j,k \leq n})(s_0)
\end{eqnarray*} 
is linear, one to one  from $\mathcal{D}_{A,s_0}$ to $\hat V_A (s_0)$, where
\begin{eqnarray}
\hat V_A(s) = \left[ - \frac{A}{s^2},\frac{A}{s^2} \right]^{1 + n} \times \left[ - \frac{A^2}{s^{p_1 + 2}},\frac{A^2}{s^{p_1 + 2}} \right]^{1 + n} \times \left[ - \frac{A^5 \ln s}{s^{p_1+2}}, \frac{A^5 \ln s}{s^{ p_1+2}}\right]^{\frac{n(n+1)}{2}}.\label{defini-of-hat-V-A}
\end{eqnarray}
 Moreover, 
 \begin{equation}\label{degree-Psi-1}
 \Psi_1 (\partial \mathcal{D}_{A,s_0}) \subset \partial \hat V_A (s_0) \text{ and } \text{deg } ( \Psi_1 \left|_{\partial \mathcal{D}_{A,s_0}} \right.)  \neq 0.
 \end{equation}

\item[$ii)$ ] In particular,  we have $(q_1, q_2) (s_0)  \in V_{A} (s_0),$ and
\begin{eqnarray*}
| q_{1,j,k}(s_0)| &\leq &\frac{A^2 \ln s_0}{2 s_0^2}, \forall j,k \leq n,\\
\left\| \frac{q_{1,-} (.,s_0)}{1 + |y|^{3}} \right\|_{L^\infty} \leq \frac{A}{2s^{2}_0} &\text{ and }&  \left\| \frac{q_{2,-} (.,s_0)}{1 + |y|^{3}} \right\|_{L^\infty} \leq \frac{A^{2} }{2s_0^{\frac{p_1 + 5}{2}}},\\
q_{1,e} (.,s_0) = 0 &\text{ and } & q_{2,e} (.,s_0) = 0.\\
\end{eqnarray*}
\end{itemize}
\end{lemma}
\begin{proof}
The proof   is straightforword and a bit length. For that reason, the proof is omitted, and we friendly refer  the  reader to          Proposition 4.5 in  \cite{TZpre15}  for    a quite  similar case.
\end{proof} 
Now, we   give   a key-proposition for our argument. More precisely, in the following proposition, we     prove an existence of a solution of equation \eqref{equation-satisfied-by-q-1-2} trapped in the shrinking set:
\begin{proposition}[Existence of a solution trapped in $V_A(s)$]\label{pro-existence-d-1-d-1}
There exists $A_2 \geq 1 $ such that for all $ A \geq A_2$ there exists $s_2(A) \geq 1$ such that for all $s_0 \geq s_2(A)$, there exists $(d_1,d_2) \in   \mathbb{R}^{ \frac{n^2 + 5n + 4}{2}} $ such that the solution $(q_1, q_2)$ of equation \eqref{equation-satisfied-by-q-1-2} with the initial data at the time $s_0,$ given by $(q_1,q_2)(s_0) = (\phi_1 , \phi_2 )(s_0)$, where $(\phi_1 , \phi_2 )(s_0)$ is defined   in Definition \ref{initial-data-profile-complex}, we have
$$(q_1, q_2)   \in V_A(s), \quad \forall s \in [s_0,+\infty).$$
\end{proposition}

 The proof is divided into 2 steps:
\begin{itemize}
\item  The first step: In this step, we    reduce   our problem  to a finite dimensional one.  In  other words,  we aim at proving that the  control of $(q_1,q_2)(s)$ in the shrinking set $V_A (s)$   reduces to  the control of the components 
$$( q_{1,0},  (q_{1,j} )_{j \leq n }  , q_{2,0},  (q_{2,j})_{j\leq n}, (q_{2,j,k} )_{j,k \leq n} )(s)$$ 
 in  $\hat V_A (s).$ 
 \item The second step:  We get the  conclusion of  Proposition \ref{pro-existence-d-1-d-1}   by using   a topological argument   in finite dimension.
\end{itemize}
\begin{proof} We here give proof of Proposition \ref{pro-existence-d-1-d-1}:

\medskip
- \textit{Step 1: Reduction to a finite dimensional problem:}
Using \textit{a priori estimates},  our problem will be reduced to the control of a finite number of  components. 
\begin{proposition}[Reduction to a  finite dimensional problem]\label{pro-reduction- to-finit-dimensional}
There exists $A_3  \geq  1$ such that for all $A  \geq A_3 $, there exists $s_3 (A) \geq 1 $ such that
for all $s_0 \geq s_3 (A).$ The  following holds:
\begin{itemize}
\item[$(a)$]   If $(q_1,q_2)(s)$ a solution 
of equation \eqref{equation-satisfied-by-q-1-2}  with  initial data at the time $s_0$ given by $(q_1,q_2)(s_0) = (\phi_1 , \phi_2 )(s_0)$  defined as  in Definition \ref{initial-data-profile-complex} with  $(d_1,d_2) \in \mathcal{D}_{A,s_0}$  defined  in Lemma   \ref{lemma-control-initial-data}.
\item[$(b)$] If  we  furthemore  assume  that $(q_1,q_2) (s) \in V_A (s)$ for all $   s \in [s_0,s_1]$ for some $s_1 \geq s_0$ and  $(q_1,q_2)(s_1) \in \partial V_A (s_1)$.
\end{itemize}

     Then, we have the following conclusions:
\begin{itemize}
\item[$(i)$] (\textit{Reduction to finite dimensions}): We have $( q_{1,0},  (q_{1,j} )_{j \leq n }  , q_{2,0},  (q_{2,j})_{j\leq n}, (q_{2,j,k} )_{j,k \leq n} )(s_1) \in \partial \hat V_A (s_1) .$
\item[$(ii)$] (\textit{Transverse outgoing crossing}) There exists  $\delta_0 > 0$ such that 
\begin{equation}\label{traverse-outgoing crossing}
\forall \delta \in (0,\delta_0), ( q_{1,0},  (q_{1,j} )_{j \leq n }  , q_{2,0},  (q_{2,j})_{j\leq n}, (q_{2,j,k} )_{j,k \leq n} )(s_1+\delta) \notin \hat V_A (s_1+\delta),
\end{equation}
which implies that $(q_1,q_2)(s_1 + \delta) \notin V_A (s_1 + \delta) $ for all $ \delta \in (0, \delta_0).$
\end{itemize}
\end{proposition}
  This proposition  makes the heart of the paper and needs many steps to be proved. For that reason, we dedicate  a whole  section to its proof (Section \ref{the proof of proposion-reduction-finite-dimensional} below). Let us admit it here, and       get to  the conclusion of     Proposition \ref{pro-existence-d-1-d-1} in the second step.
 
 \medskip
\textit{ - Step 2: Conclusion of Proposition \ref{pro-existence-d-1-d-1} by a topological argument.}
In this step, we  finish the proof  of   Proposition \ref{pro-existence-d-1-d-1}. In  fact,   we aim at  proving  the existence of a parameter $(d_1,d_2) \in \mathcal{D}_{A,s_0}$ such that  the solution  $(q_1,q_2)(s)$ of equation \eqref{equation-satisfied-by-q-1-2}   with   initial  data $(q_1 ,q_2 )(s_0) = (\phi_1 , \phi_2 )(s_0),$ exists globally  for all $s \in [s_0, + \infty)$ and satisfies
$$ (q_1, q_2)(s)  \in V_A (s).$$
Our  argument is  analogous  to the   argument of  Merle and Zaag in \cite{MZdm97}. For that reason,   we only give  a brief proof. Let us fix $K, A, s_0$  such that Lemma  \ref{lemma-control-initial-data} and Proposition \ref{pro-reduction- to-finit-dimensional} hold. We first  consider  $(q_1, q_2)_{d_1,d_2}(s), s \geq s_0$  a solution of equation \eqref{equation-satisfied-by-q-1-2} with initial data at  $s_0$  is  $(q_1,q_2)(s_0)$, which depend on $(d_1,d_2)$ as follows
$$ (q_1,q_2)_{d_1,d_2}(s_0) =  (\phi_1, \phi_2) (s_0).$$
From Lemma  \ref{lemma-control-initial-data}  and  by construction  of the set  $\mathcal{D}_{A,s_0},$ we know that 
\begin{equation}\label{q-1-q-2-s-0-V-A-s-0}
(q_1,q_2) (s_0) \in V_A (s_0).
\end{equation} 
By contradiction, we assume  that  for all $(d_1,d_2) \in \mathcal{D}_{A, s_0}$ there exists  $s_1  \in [s_0,  + \infty)$ such that
$$  (q_1, q_2)_{d_1,d_2} (s_1) \notin V_A(s_1).$$
Then,  for each  $(d_1, d_2) \in \mathcal{D}_{A,s_0},$ we can define   
$$s^* (d_1, d_2) = \inf \{ s_1  \geq s_0 \text{ such that } (q_1,q_2)_{d_1,d_2} (s_1) \notin V_A (s_1)\}.$$
Since   there exists   $s_1$ such that $(q_1,q_2)(s_1) \notin V_A(s_1)$   we deduce that  $s^*(d_1,d_2) < + \infty$ for 	all $ (d_1,d_2) \in \mathcal{D}_{A, s_0}.$ Besides that,   using    \eqref{q-1-q-2-s-0-V-A-s-0}, and the minimality   of $s^* (d_1,d_2),$ the continuity of $(q_1,q_2)$ in $s$ and  the closeness  of $V_A(s)$  we derive that $ (q_1,q_2) (s^*(d_1,d_2)) \in \partial V_A (s^*(d_1,d_2))$ and for all  $s \in [s_0, s^*(d_1,d_2)],$
$$ (q_1,q_2) (s) \in V_A (s).$$

\noindent
Therefore,  from  item $(i)$ of Proposition \ref{pro-reduction- to-finit-dimensional} we see that  
$$ ( q_{1,0},  (q_{1,j} )_{j \leq n }  , q_{2,0},  (q_{2,j})_{j\leq n}, (q_{2,j,k} )_{j,k \leq n} )(s^*(d_1,d_2))\in \hat V_A(s^*(d_1,d_2)).$$
This means that following mapping $\Gamma$ is well-defined: 
\begin{eqnarray*}
\Gamma : \mathcal{D}_{A, s_0} & \to & \partial \left(  [-1, 1]^{1 + n} \times [-1, 1]^{1 + n}  \times [-1, 1]^{\frac{n (n+1)}{2}}\right)\\
(d_1,d_1)  &\mapsto  & \left(  \frac{s_*^2(d_1,d_2)}{A}(q_{1,0}, (q_{1,j})_{j\leq n}), \frac{s^{p_1 + 2}}{A^2}  (q_{2,0}, (q_{2,j})_{j\leq n}),
 \frac{s_*^{p_1+2}(d_1,d_2)}{A^5 \ln s_*(d_1,d_2)}  (q_{2,j,k})_{j,k \leq n }\right)(s^*(d_1,d_2)).
\end{eqnarray*}
Moreover, it       satisfies  the two  following properties:
\begin{itemize}
\item[(i)] $\Gamma$ is continuous from $\mathcal{D}_{A, s_0}$ to $ \partial \left(  [-1, 1]^{\frac{n^2 + 5n + 4}{2}}\right).$  This is a consequence  of item $(ii)$ in Proposition \eqref{pro-reduction- to-finit-dimensional}.
\item[(ii)] The degree  of the   restriction $\Gamma \left|\right. _{\partial \mathcal{D}_{A,s_0}}$ is   non zero.  
Indeed, again  by item $(ii)$ in Proposition \ref{pro-reduction- to-finit-dimensional}, we have 
$$ s^* (d_1,d_2) = s_0,$$
in this case. Applying  \eqref{degree-Psi-1}, we get the conclusion.
\end{itemize}
In fact, such  a mapping $\Gamma$  can not  exist by Index theorem,   this is a contradiction. Thus,   Proposition \ref{pro-existence-d-1-d-1}  follows, assuming  that  Proposition  \ref{pro-reduction- to-finit-dimensional} (see Section  \ref{the proof of proposion-reduction-finite-dimensional} for the proof of latter)
\end{proof}
\subsection{ The proof of Theorem \ref{Theorem-profile-complex}} 
In this section, we aim at giving the  proof of Theorem \ref{Theorem-profile-complex}.
\begin{proof} \textit{ Proof of Theorem \ref{Theorem-profile-complex} assuming  that Proposition  \ref{pro-reduction- to-finit-dimensional}}

\medskip
+ \textit{The proof of  item $(i)$ of Theorem \ref{Theorem-profile-complex}:} 
 Using Proposition \ref{pro-existence-d-1-d-1},  there exists      initial data  $(q_1,q_2)_{d_1,d_2}(s_0) = (\phi_1,\phi_2)(s_0)$ 	 such that the solution of equation of  \eqref{equation-satisfied-by-q-1-2} exists  globally  on $[s_0, + \infty)$ and   satisfies:  
 $$ (q_1,q_2) (s) \in V_A (s), \forall  s \in [s_0, + \infty),$$
 Thanks to  similarity variables \eqref{similarity-variales},  \eqref{defini-q-1-2}   and item   $(i)$   in Lemma \ref{lemma-estiam-q-1-2-in V-A},   we conclude that there exist   initial data $u^0 $  of the form given in Remark \ref{remark-initial-data} with $(d_1,d_2)$  given in Proposition \ref{pro-existence-d-1-d-1}  such that  the solution  $u(t)$  of equation  \eqref{equ:problem}  exists  on $[0,T),$ where  $T  = e^{- s_0}$ and  satisfies \eqref{esttima-theorem-profile-complex} and \eqref{estima-the-imginary-part}. Using these two estimates, we see  that 
$$  u(0,t)  \sim   \kappa (T -t)^{-\frac{1}{p-1}}  \text{ as }  t \to T,$$ 
which means  that  $u$ blows up  at time $T$ and   the origin is a blowup point.    It remains  to prove that  for all $x \neq 0,$  $x$ is not a blowup point of $u$. The following Lemma  allows us  to conclude. 
 \begin{lemma}[No blow up  under  some  threshold]\label{lemma-no-blowup-solution}
For all $C_0 > 0, 0 \leq  T_1 <  T$ and $\sigma>0$ small enough,  there exists $\epsilon_0 (C_0,T, \sigma)> 0$  such that   $u(\xi, \tau)$ satisfies the following  estimates    for all $|\xi| \leq   \sigma, \tau \in \left[T_1,T\right)$:
 $$ \left| \partial_\tau u - \Delta u  \right| \leq C_0 |u|^p,$$
 and 
 $$ |u(\xi,\tau)|       \leq \epsilon_0 (1 -\tau)^{-\frac{1}{p-1}}.$$
 Then,  $u$ does not  blow up at $\xi = 0, \tau = T$. 
 \end{lemma}
 \begin{proof}
  The   proof of this Lemma is processed  similarly   to      Theorem   2.1     in \cite{GKcpam89}.        Although  the proof  of  \cite{GKcpam89} was  given in the  real case, it extends naturally  to the complex  valued  case.    
 \end{proof}
 
 \noindent We next use   Lemma  \ref{lemma-no-blowup-solution}  to conclude that $u$ does not blow up at  $x_0 \neq  0.$  Indded,  if   $x_0 \neq 0$ we use  \eqref{esttima-theorem-profile-complex}   to deduce the following: 
 \begin{equation}\label{estima-u-x-0-neq-0}
 \sup_{|x - x_0| \leq \frac{|x_0|}{2}}   (T - t)^{\frac{1}{p -1}} | u(x,t) |  \leq  \left|  f_0 \left(   \frac{ \frac{|x_0|}{2}}{  \sqrt{{(T -t )}|\ln (T -t)| } }\right)  \right| + \frac{C}{ \sqrt{ | \ln (T - t)|}} \to 0, \text{ as }  t \to T. 
 \end{equation}
 Applying Lemma  \ref{lemma-no-blowup-solution}  to $u(x - x_0, t),$ with some $\sigma$ small enough such that $\sigma \leq \frac{|x_0|}{2},$ and $T_1$  close enough to $T,$ we see that $u(x - x_0, t)$ does not blow up at time   $T$ and  $x = 0$. Hence $x_0 $ is not a blow-up point of  $u$. This concludes  the proof of item $(i)$ in Theorem \ref{Theorem-profile-complex}.  
 
 \medskip
+  \textit{The proof of item $(ii)$ of Theorem \ref{Theorem-profile-complex}:}
 Here,  we   use the argument  of Merle in   \cite{Mercpam92}    to deduce  the existence of  $u^* = u_1^* + i u_2^*$   such that   $u(t) \to u^* $  as  $t \to T$ uniformly on compact sets of $\mathbb{R}^n \backslash \{0\}$. In addition to that, we use   the techniques in  Zaag \cite{Zcpam01}, Masmoudi and Zaag  \cite{ MZjfa08}, Tayachi and Zaag \cite{TZpre15} for the proofs of \eqref{asymp-u-start-near-0-profile-complex} and \eqref{asymp-u-start-near-0-profile-complex-imaginary-part}. 
 
\noindent 
 Indeed,   for all   $x_0 \in \mathbb{R}^n , x_0 \neq 0 $, we deduce   from   \eqref{esttima-theorem-profile-complex},  \eqref{estima-the-imginary-part}   that  not only  \eqref{estima-u-x-0-neq-0} holds but also the following satisfied:
 \begin{eqnarray}
  \sup_{|x - x_0| \leq \frac{|x_0|}{2}}   (T - t)^{\frac{1}{p -1}} |\ln(T -t)|| u_2(x,t) |  &\leq & \left|  \frac{3 |x_0|^2}{2 (T -t ) |\ln (T -t)|    } f_0 ^p\left(  \frac{ \frac{|x_0|}{2}}{  \sqrt{{(T -t )}|\ln (T -t)| } }\right)  \right|\label{estimates-T-t-u-leq-epsilon}  \\
  &+&  \frac{C}{  | \ln (T - t)|^{\frac{p_1}{2}}} \to 0,\nonumber \text{ as } t \to T.
 \end{eqnarray}
  We now consider  $x_0$ such that $|x_0|$ is  small enough,   and $K_0$ to be fixed later. We   define  $t_0(x_0)$  by 
  \begin{equation}\label{x-0-leq-delta-t-x-0}
  |x_0|   =  K_0 \sqrt{ T -t_0(x_0) |\ln (T -t_0(x_0))|} . 
  \end{equation}
Note that  $t_0 (x_0)$  is unique  when  $|x_0|$  is small enough and $t_0 (x_0)\to T$  as  $x_0 \to 0$.  We  introduce  the rescaled functions  $U(x_0, \xi, \tau)$  and   $V_2 (x_0, \xi, \tau)$        as follows:
   \begin{equation}\label{equa-upsilon-xi-tau}
U (x_0, \xi, \tau)  = \left( T - t_0 (x_0)\right)^{\frac{1}{p-1}} u(x,t).
\end{equation}
and  
\begin{equation}\label{defini-V-2-x-0-xi-tau}
V_2 (x_0, \xi, \tau)  = |\ln (T- t_0(x_0))| U_2 (x_0, \xi, \tau),
\end{equation}
where  $U_2 (x_0, \xi, \tau)$   is  defined by
$$ U (x_0, \xi, \tau)  = U_1 (x_0, \xi, \tau) + i U_2 (x_0, \xi, \tau),$$
and  
\begin{equation}\label{relation-x-and-xi-tau}
(x,t) = \big(x_0 + \xi\sqrt{T - t_0(x_0)},t_0(x_0) + \tau (T - t_0(x_0))\big), \text{ and } (\xi, \tau) \in \mathbb{R}^n \times \left[ - \frac{t_0(x_0)}{T - t_0(x_0)}, 1 \right).
\end{equation}
We can see that  with these notations, we derive from item $(i)$ in Theorem \ref{Theorem-profile-complex} the following estimates for initial data at $\tau = 0 $ of $U$ and $V_2$
\begin{eqnarray}
\sup_{|\xi| \leq   |\ln(T - t_0(x_0))|^{\frac{1}{4}}} \left| U(x_0, \xi, 0) - f_0(K_0)\right| & \leq & \frac{C}{ 1 + (|\ln(T - t_0(x_0))|^{\frac{1}{4}})} \to 0 \quad \text{ as } x_0 \to 0,\label{condition-initial-K-0-f-0}\\
\sup_{|\xi| \leq  |\ln(T - t_0(x_0))|^{\frac{1}{4}}} \left| V_2(x_0, \xi, 0)- g_0(K_0)\right| &\leq &\frac{C}{ 1 + (|\ln(T - t_0(x_0))|^{ \gamma_1})} \to 0 \quad \text{ as } x_0 \to 0.\label{condition-initial-K-0-g-0}
\end{eqnarray}
where $f_0 (x), g_0 (x)$ are defined as in  \eqref{defini-f-0} and \eqref{defini-g-0-z} respectively, and  $\gamma_1 = \min \left(  \frac{1}{4} ,   \frac{p_1}{2} \right) $. Moreover, using equations \eqref{equation-satisfied-u_1-u_2}, we derive the following equations for $U, V_2$: for all   $\xi \in \mathbb{R}^n, \tau \in\left[ 0, 1 \right)$ 
\begin{eqnarray}
\partial_\tau U =  \Delta_\xi  U +  U^p,\label{equa-U-x-0-xi-tau}\\
\partial_\tau V_2 =  \Delta_\xi V_2 + V_2 G_2 (U_1, U_2), \label{equa-V-2-x-0-xi-tau}
\end{eqnarray}
where $G$   is defined  by 
\begin{equation}\label{defini-G-2-U-1-2}
G (U_1, U_2) U_2  = F_2 (U_1,U_2),
\end{equation}
and $F_2$ is defined in \eqref{defi-mathbb-A-1-2}. We note that  $G_2, F_2$ are polynomials   of $U_1, U_2$.

 Besides that,  from  \eqref{estimates-T-t-u-leq-epsilon} and \eqref{equa-U-x-0-xi-tau}, we can apply Lemma  \ref{lemma-no-blowup-solution}   to $U$ when $|\xi| \leq |\ln (T - t_0(x_0))|^{\frac{1}{4}}$  and obtain:
\begin{equation}\label{bound-U-xi-tau-x-0}
\sup_{|\xi| \leq \frac{1}{2}|\ln (T -t_0(x_0))|^\frac{1}{4}, \tau \in [0,1) }  |U (x_0, \xi,\tau)|     \leq C. 
\end{equation}
and we  aim at  proving    for $V_2 (x_0, \xi, \tau)$  that
\begin{equation}\label{bound-V-2-xi-tau-x-0}
\sup_{|\xi| \leq  \frac{1}{16}|\ln (T -t_0(x_0))|^\frac{1}{4}, \tau \in [0,1) }  |V_2(x_0, \xi,\tau)|     \leq C.
\end{equation}
+  \textit{The proof for \eqref{bound-V-2-xi-tau-x-0}:}  We first use \eqref{bound-U-xi-tau-x-0} to derive  the following rough estimate:
\begin{equation}\label{estima-V-2-1-step-0}
\sup_{|\xi| \leq \frac{1}{2} |\ln (T -t_0(x_0))|^\frac{1}{4}, \tau \in [0,1) }  |V_2(x_0, \xi,\tau)|     \leq C |\ln(T -t_0(x_0))|.
\end{equation}
We first introduce $\psi (x)$ a cut-off function $\psi \in C^\infty_0 (\mathbb{R}^n), 0 \leq  \psi \leq 1, supp(\psi ) \subset B(0,1), \psi = 1   $ on $B( 0, \frac{1}{2}).$  We   introduce
\begin{equation}\label{U-2-1-psi-1-x-0}
\psi_1 (\xi) = \psi \left(  \frac{2\xi}{  |\ln (T -t_0 (x_0))|^{\frac{1}{4}}} \right) \text{ and  }  V_{2,1} (x_0, \xi, \tau) = \psi_1 (\xi) V_2 (x_0,\xi, \tau).
\end{equation}
Then, we deduce from  \eqref{equa-V-2-x-0-xi-tau}  an  equation satisfied by $V_{2,1}$
\begin{equation}\label{equa-V-2-1-x-0}
\partial_\tau  V_{2,1}  =  \Delta_\xi V_{2,1} - 2 \text{ div} (V_2 \nabla \psi_1)  + V_2 \Delta \psi_1 + V_{2,1} G_1(U_1,U_2).
\end{equation}
Hence, we can write $V_{2,1} $ with a integral  equation as follows
\begin{equation}\label{equa-integral-V-2-1}
V_{2,1} (\tau)  =   e^{\Delta \tau} (V_{2,1}(0)) +  \int_0^\tau e^{(\tau - \tau')\Delta} \left(  - 2 \text{ div } (V_2 \nabla \psi_1) + V_2 \Delta\psi_1 + V_{2,1} G(U_1, U_2)(\tau')  \right) d \tau'.
\end{equation}
Besides that, using \eqref{bound-U-xi-tau-x-0} and \eqref{estima-V-2-1-step-0} and  the fact  that 
\begin{eqnarray*}
| \nabla \psi_1|  \leq  \frac{C}{ | \ln (T -t_0(x_0))|^{\frac{1}{4}}}, 
| \Delta \psi_1|  \leq  \frac{C}{ | \ln (T -t_0(x_0))|^{\frac{1}{2}}},
\end{eqnarray*}
we deduce that
\begin{eqnarray*}
\left|   \int_0^\tau e^{(\tau - \tau')\Delta} \left(  - 2 \text{ div } (V_2 \nabla \psi_1) \right) d\tau' \right| &\leq & C \int_{0}^\tau \frac{\| V_2 \nabla \psi_1\|_{L^\infty} (\tau')}{ \sqrt { \tau - \tau'}}  d\tau'       \leq C |\ln (T - t_0 (x_0))|^{\frac{3}{4}},\\
\left|   \int_0^\tau e^{(\tau - \tau')\Delta} \left(  V_2 (\tau') \Delta \psi_1  \right) d\tau' \right| &\leq & C  \int_0^\tau  \| V_2 \Delta \psi_1\|_{\infty}   (\tau') d\tau' \leq C |\ln (T -t_0 (x_0))|^\frac{1}{2}, \\
\left|   \int_0^\tau e^{(\tau - \tau')\Delta} \left(  V_2 \psi_1 G(U_1, U_2)(\tau')   \right) d\tau' \right|   &\leq  & C \int_0^\tau \| V_{2,1 } G_2 (U_1,U_2)\|_{L^\infty} (\tau') d\tau'.
\end{eqnarray*}
Note that $G_2 (U_1, U_2)$  in the last line is bounded  on $|\xi| \leq |\ln(T -t_0)|^\frac{1}{4}, \tau \in [0,1)$  because  it is  a polynomial in $U_1, U_2$ and        \eqref{bound-U-xi-tau-x-0} holds,  then,  we  derive 
$$ \| V_{2,1}  G_2 (U_1, U_2)\|_{L^\infty}(\tau')  \leq C \| V_{2,1}\|_{L^\infty} (\tau').$$
Hence, from  \eqref{equa-integral-V-2-1} and the above estimates, we derive
$$ \| V_{2,1}(\tau)\|_{L^\infty} \leq C | \ln (T -t_0 (x_0)) |^{\frac{3}{4}} +  C \int_0^\tau   \|V_{2,1}(\tau')\|_{L^\infty} d \tau'. $$
Thanks to Gronwall Lemma,  we deduce that 
$$ \|V_{2,1} (\tau)\|_{L^\infty}  \leq C |\ln(T -t_0(x_0))|^{\frac{3}{4}}, \forall  \tau \in [0,1),$$
which yields
\begin{equation}\label{estima-V-2-1-step-1}
\sup_{|\xi| \leq \frac{1}{4} |\ln (T -t_0(x_0))|^\frac{1}{4}, \tau \in [0,1) }  |V_2(x_0, \xi,\tau)|     \leq C |\ln(T -t_0(x_0))|^{\frac{3}{4}}.
\end{equation}
We apply  iteratively for 
$$ V_{2,2} (x_0, \xi, \tau) =   \psi_2 (\xi) V_{2} (x_0,\xi, \tau)  \text{ where }  \psi_2 (\xi)  = \psi \left( \frac{4 \xi}{ |\ln(T -t_0 (x_0))|^{\frac{1}{4}}} \right).$$
Similarly, we  deduce that
$$ \sup_{|\xi| \leq \frac{1}{8} |\ln (T -t_0(x_0))|^\frac{1}{4}, \tau \in [0,1) }  |V_2(x_0, \xi,\tau)|     \leq C |\ln(T -t_0(x_0))|^{\frac{1}{2}}.$$
We apply this  process a  finite   number of  steps to  obtain \eqref{bound-V-2-xi-tau-x-0}.  We now come back to our problem, and aim at proving that:  
\begin{eqnarray}
\sup_{|\xi| \leq \frac{1}{16}  |\ln(T - t_0(x_0))|^{\frac{1}{4}}, \tau \in [0,1)} \left| U (x_0, \xi, \tau) - \hat U_{K_0} (\tau) \right| &\leq & \frac{C}{ 1 + |\ln (T- t_0(x_0) )|^{\gamma_2}}, \label{sup-v-xi-tau-apro-1}\\
\sup_{|\xi| \leq \frac{1}{32}|\ln(T - t_0(x_0))|^{\frac{1}{4}}, \tau \in [0,1)} \left| V_2 (x_0, \xi, \tau) - \hat V_{2,K_0} (\tau) \right| &\leq & \frac{C}{ 1 + |\ln (T- t_0(x_0) )|^{\gamma_3}}, \label{sup-V-2-xi-tau-apro-1}
\end{eqnarray}
where  $\gamma_2, \gamma_3$ are positive small enough  and   $( \hat U_{K_0} , \hat V_{2,K_0}) (\tau)  $  is  the  solution of  the following system:
\begin{eqnarray}
\partial_\tau \hat U_{K_0} &=& \hat U_{K_0}^p,\label{ODE-hat-U-K-0}\\
 \partial_\tau \hat V_{2, K_0} &=& p \hat U_{K_0}^{p-1} \hat V_{2,K_0}\label{ODE-hat-U-K-0}. 
\end{eqnarray}
with initial data  at $\tau = 0$
\begin{eqnarray*}
\hat U_{K_0} (0) &=& f_0 (K_0),\\
\hat V_{2,K_0} (0) &=& g_0 (K_0). 
\end{eqnarray*}
given by 
\begin{eqnarray}
\hat U_{K_0} (\tau) &=& \left(  (p-1) (1 - \tau)  + \frac{(p-1)^2 K_0^2}{4 p}\right)^{-\frac{1}{p-1}}  ,\label{defini-hat-U-K-0-tau}\\
\hat V_{2,K_0} (\tau) &=& K_0 ^2 \left( (p-1) (1 - \tau) + \frac{(p-1)^2 K_0^2}{4 p}\right)^{-\frac{p}{p-1}}.  \label{defini-hat-V-2-K-0-tau}
\end{eqnarray}
for all $\tau \in [0,1)$.  The proof of \eqref{sup-v-xi-tau-apro-1}  is cited to Section 5 of Tayachi and Zaag \cite{TZpre15} and the proof of  \eqref{sup-V-2-xi-tau-apro-1} is similar.   For the  reader's  convenience, we give it here.  Let us consider 
\begin{equation}\label{defini-mathcal-V-2}
\mathcal{V}_2 = V_2 - \hat V_{2,K_0} (\tau).
\end{equation}
Then, $\mathcal{V}_2$ satisfies
\begin{equation}\label{estima-mathcal-V-2}
\sup_{|\xi| \leq \frac{1}{ 16} | \ln (T - t_0(x_0))|^{\frac{1}{4}}, \tau \in [0,1)} | \mathcal{V}_{2}| \leq C.
\end{equation}
 We use  \eqref{equa-V-2-x-0-xi-tau} to derive  an equation  on  $\mathcal{V}_2$ as follows:
\begin{equation}\label{equat-mathcal-V-2}
\partial_\tau \mathcal{V}_2 = \Delta \mathcal{V}_2 +  p \hat U_{K_0}^{p-1} \mathcal{V}_2 +  p (U_1^{p-1}  - \hat U_{K_0}^{p-1} ) V_2 + \mathcal{G}_2(x_0, \xi, \tau),
\end{equation}
where 
$$ \mathcal{G}_2 (x_0, \xi,\tau) =  V_2 [ G_2 (U_1, U_2) - p U_1^{p-1} ].$$
Note that, from definition of $G_2$  and  \eqref{bound-U-xi-tau-x-0} we deduce that 
$$  \sup_{|\xi| \leq  \frac{1}{2} |\ln(T -t_0)|^{\frac{1}{4}}, \tau \in [0,1) } |  G_2 (U_1, U_2)  - p U_1^{p-1}| \leq C |U_2|,$$
Hence,  using   \eqref{defini-V-2-x-0-xi-tau} and \eqref{bound-V-2-xi-tau-x-0}   and   we derive 
\begin{equation}\label{estima-mathcal-G-2}
  \sup_{|\xi| \leq  \frac{1}{16} |\ln(T -t_0)|^{\frac{1}{4}}, \tau \in [0,1) } | \mathcal{G}_2 (x_0, \xi, \tau)| \leq   \frac{C}{  | \ln (T -t_0 (x_0))|}.
\end{equation}
We also define 
$$\bar{\mathcal{V} }_2 = \psi_* (\xi) \mathcal{V}_2,$$
where 
$$ \psi_* = \psi \left( \frac{1 6 \xi}{ |\ln (T- t_0 (x_0))|^{\frac{1}{4}}}\right),$$
and $ \psi  $ is the cut-off function  which has been introduced  above. We also note that $\nabla \psi_*, \Delta \psi_*$ satisfy the following estimates
\begin{equation}\label{estima-psi-nabla-xi-psi-x-0}
\| \nabla_\xi  \psi_* \|_{L^\infty} \leq \frac{C}{ |\ln (T- t_0 (x_0))|^{\frac{1}{4}}}  \text{ and }  \| \Delta_\xi  \psi_*\|_{L^\infty} \leq \frac{C}{ |\ln (T- t_0 (x_0))|^{\frac{1}{2}}}.
\end{equation}
In particular,  $\bar{\mathcal{V}}_2$ satisfies
\begin{equation}\label{euqa-bar-mathcal-V-2}
\partial_\tau   \bar {\mathcal{V}}_2  =  \Delta \bar {\mathcal{V}}_2 + p \hat U_{K_0}^{p-1} (\tau)  \bar {\mathcal{V}}_2   - 2 \text{ div } (\mathcal{V}_2 \nabla \psi_*) + \mathcal{V}_2 \Delta \psi_* + p (U_1^{p-1} - \hat U_{K_0}^{p-1}) \psi_* V_2 + \psi_* \mathcal{G}_2,
\end{equation}

By Duhamel principal,  we derive the following integral  equation 
\begin{equation}\label{duhamel-bar-mathcal-V-2}
\bar{ \mathcal{V}}_2 (\tau)  = e^{\tau \Delta} (\bar{ \mathcal{V}}_2 (\tau)  ) + \int_0^\tau  e^{(\tau - \tau')\Delta} \left( p \hat U_{K_0}^{p-1}   \bar {\mathcal{V}}_2   - 2 \text{ div } (\mathcal{V}_2 \nabla \psi_*) + \mathcal{V}_2 \Delta \psi_* + p (U_1^{p-1} - \hat U_{K_0}^{p-1}) \psi_* V_2 + \psi_* \mathcal{G}_2   \right) (\tau') d\tau'.
\end{equation}
Besides that, we use    \eqref{sup-v-xi-tau-apro-1},  \eqref{defini-hat-U-K-0-tau},   \eqref{estima-mathcal-V-2},   \eqref{estima-psi-nabla-xi-psi-x-0},   \eqref{estima-mathcal-G-2} to derive the  following estimates: for all $\tau \in [0,1)$ 
\begin{eqnarray*}
|\hat U_{K_0} (\tau )|  & \leq & C ,\\
\| \mathcal{V}_2 \nabla \psi_* \|_{L^\infty} (\tau)  & \leq  & \frac{C }{ |\ln (T -t_0(x_0))|^{\frac{1}{4}}},\\
\| \mathcal{V}_2 \Delta \psi_* \|_{L^\infty} (\tau)  & \leq  & \frac{C }{ |\ln (T -t_0(x_0))|^{\frac{1}{2}}},\\
\left\|  \left(U_1^{p-1} - \hat U_{K_0}^{p-1} \right) \psi_*   \right\|_{L^\infty} (\tau)& \leq & \frac{C}{ | \ln(T -t_0(x_0))|^{\gamma_2}},  \\
\| \mathcal{G}_2 \psi_*\|_{L^{\infty}} & \leq & \frac{C}{ |\ln (T -t_0(x_0))| }.
\end{eqnarray*}
where $\gamma_2$ given in \eqref{sup-v-xi-tau-apro-1}. Hence, we  derive from the above estimates that:  for all $\tau \in [0,1)$
\begin{eqnarray*}
|  e^{(\tau - \tau')\Delta}p \hat U_{K_0}^{p-1}   \bar {\mathcal{V}}_2  (\tau') |  & \leq  &C \|\bar {\mathcal{V}}_2  (\tau') \|,\\
| e^{(\tau - \tau')\Delta} (\text{div} (\mathcal{V}_2 \nabla \psi_*)) | &\leq & C \frac{1}{\sqrt{ \tau - \tau'}}  \frac{1}{| \ln (T-  t_0(x_0))|^{\frac{1}{4}}}  ,\\
|e^{(\tau - \tau')\Delta } ( \mathcal{V}_2 \Delta \psi_*)  |   & \leq &   \frac{C}{|\ln(T - t_0(x_0) )|^{\frac{1}{2}} },\\
|e^{(\tau - \tau')\Delta } ( p (U_1^{p-1} - \hat U_{K_0}^{p-1}) \psi_* V_2   )(\tau')  |   & \leq &   \frac{C}{ |\ln (T -t_0(x_0))|^{\gamma_2}},\\
|e^{(\tau - \tau')\Delta }  (\psi_* \mathcal{G}_2  )(\tau')| &\leq & \frac{C}{|\ln (T - t_0(x_0))|}.
\end{eqnarray*}
   Pluggin    into   \eqref{duhamel-bar-mathcal-V-2}, we obtain
$$  \|\bar{\mathcal{V}}_2 (\tau)\|_{L^\infty}  \leq  \frac{C}{ |\ln(T -t_0(x_0))|^{\gamma_3}} +  C \int_{0}^\tau  \|\bar{\mathcal{V}}_2 (\tau')\|_{L^\infty}   d \tau' ,$$
where $\gamma_3 = \min (\frac{1}{4}, \gamma_2)$. Then, thanks to Gronwall inequality, we get
$$\| \bar{ \mathcal{V}}_2\|_{L^\infty} \leq \frac{C}{|\ln(T -t_0(x_0))|^{\gamma_3}}.$$
Hence,  \eqref{sup-V-2-xi-tau-apro-1} follows . Finally, we  easily  find the asymptotics  of $u^*$ and $u_2^*$ as follows, thanks to the definition of $U$ and  $V_2$  and   to estimates \eqref{sup-v-xi-tau-apro-1} and   \eqref{sup-V-2-xi-tau-apro-1}:
  \begin{equation}\label{limit-u-start}
u^* (x_0) = \lim_{t \to T} u(x_0, t) = (T- t_0 (x_0))^{- \frac{1}{p-1}} \lim_{\tau \to 1} U (x_0, 0, \tau) \sim (T- t_0 (x_0))^{- \frac{1}{p-1}}  \left(\frac{(p -  1)^2}{ 4 p} K_0^2 \right)^{- \frac{1}{p-1}},
\end{equation} 
and 
\begin{equation}\label{limit-im-u-*}
u_2^* = \lim_{t \to T} u_2(x_0, t) = \frac{(T- t_0 (x_0))^{- \frac{1}{p-1}}}{|\ln (T- t_0 (x_0))|} \lim_{\tau \to 1} V_2 (x_0, 0, \tau) \sim \frac{(T- t_0 (x_0))^{- \frac{1}{p-1}}}{|\ln (T- t_0 (x_0))|} \left(\frac{(p -  1)^2}{ 4 p} \right)^{- \frac{p}{p-1}} (K_0^2)^{- \frac{1}{p-1}}.
\end{equation}
Using the relation \eqref{x-0-leq-delta-t-x-0}, we find that
\begin{equation}\label{asymp-T-t-0-x-0} 
  T  - t_0 \sim \frac{|x_0|^2}{ 2 K_0^2 |\ln |x_0||} \text{ and  }\ln(T - t_0(x_0)) \sim 2 \ln (|x_0|), \quad \text{ as } x_0 \to 0.
  \end{equation}
Plugging  \eqref{asymp-T-t-0-x-0}  into   \eqref{limit-u-start} and   \eqref{limit-im-u-*}, we get the conclusion of  item $(ii)$ of Theorem \ref{Theorem-profile-complex}. 

This concludes the proof of Theorem \ref{Theorem-profile-complex} assuming that Proposition \ref{pro-reduction- to-finit-dimensional} holds. Naturally, we need to prove this propostion  on order  to finish the argument. This will be done  in the next section.
\end{proof}
 \section{The proof of Proposition \ref{pro-reduction- to-finit-dimensional}}\label{the proof of proposion-reduction-finite-dimensional}
 This section is devoted to the proof of Proposition \ref{pro-reduction- to-finit-dimensional},
 which is the heart of our analysis. We proceed into two parts. In the first part, we derive \textit{a priori estimates} on $q(s)$ in $V_A(s)$. In the second part, we show that the new bounds are better than those defined in $V_A(s)$, except for the first  components $( q_{1,0},  (q_{1,j} )_{j \leq n }  , q_{2,0},  (q_{2,j})_{j\leq n}, (q_{2,j,k} )_{j,k \leq n} )(s)$. This means that the problem is reduced to the control of these   components, which is the conclusion of item $(i)$ of Proposition \ref{pro-reduction- to-finit-dimensional}. Item $(ii)$ of Proposition \ref{pro-reduction- to-finit-dimensional} is just a direct consequence of the dynamics of these modes. Let us start the first part.

\subsection{A priori estimates on $(q_1,q_2)$ in $V_A(s)$.}
In this subsection, we aim at proving  the following proposition: 
 \begin{proposition}\label{prop-dynamic-q-1-2-alpha-beta} 
 There exists $A_4\geq 1,$ such that for all $ A \geq A_4$ there 
 exists $s_4(A)\geq 1$, such that the following holds for all
  $s_0 \geq s_4(A)$: we assume that for all $s \in [\sigma,s_1], (q_1,q_2)(s) \in V_A(s)$  for some $s_1\geq  s_0$. Then, the following holds for all $s \in [s_0,s_1]$: 
\begin{itemize}
\item[$(i)$] (\textit{ODE satisfied  by the positive modes})  For all $j \in \{1,n\}$ we have
\begin{equation}\label{ODE-q-1-0-1}
\left| q_{1,0}' (s) -  q_{1,0} (s) \right| + \left| q_{1,j}' (s) -  \frac{1}{2}  
q_{1,j} (s) \right|  \leq \frac{C}{s^2},\forall j\leq n.
\end{equation}
\begin{equation}\label{ODE-Phi-0-1}
\left| q_{2,0}' (s) -  q_{2,0}(s)  \right| +  \left| q_{2,j}' (s) -  \frac{1}{2}   q_{2,j}(s)  \right|  \leq \frac{C }{s^{p_1 +2}}, \forall j \leq n.
\end{equation}
\item[$(ii)$] (\text{ODE satisfied by the null modes}) For all $j,k \leq n$
\begin{equation}\label{ODE-q-1-2}
 \left| q_{1,j,k}' (s) + \frac{2}{s} q_{1,j,k} (s) \right|  \leq \frac{C A}{s^3},
\end{equation}
\begin{equation}\label{ODE-Phi-2}
\left| q_{2,j,k} '(s) + \frac{2}{s} q_{2,j,k}(s)\right| \leq \frac{C A^2 \ln s}{s^{p_1^* + 3}}.
\end{equation}
\item[$(iii)$] (\textit{Control the  negative part})
\begin{equation}\label{Estimata-q-1--}
\left\| \frac{q_{1,-}(.,s)}{1 + |y|^{3}}\right\|_{L^\infty} \leq  Ce^{- \frac{s - \tau }{2}} \left\| \frac{q_{1,-}(.,\tau)}{1 + |y|^{3}}\right\|_{L^\infty}  +  C \frac{ e^{- (s - \tau )^2}}{s^{\frac{3}{2}}}  \|q_{1,e}(.,\tau)\|_{L^\infty} + \frac{C (1 + s -\tau)}{s^2}, 
\end{equation}
\begin{equation}\label{Estimat-q-2--}
\left\| \frac{q_{2,-}(.,s)}{1 + |y|^{3}}\right\|_{L^\infty} \leq  Ce^{- \frac{s - \tau }{2}} \left\| \frac{q_{2,-}(.,\tau)}{1 + |y|^{3}}\right\|_{L^\infty}  +  C \frac{ e^{- (s - \tau )^2}}{s^{\frac{3}{2}}}  \|q_{2,e}(.,\tau)\|_{L^\infty} + \frac{C (1 + s -\tau)}{s^{\frac{p_1 + 5}{2}}}.
\end{equation}
\item[$(v)$] (\textit{Outer part}) 
\begin{equation}\label{outer-Q-e}
\left\| q_{1,e} (.,s) \right\|_{L^\infty} \leq  C e^{- \frac{(s -\tau)}{p} } \|q_{1,e}(.,\tau)\|_{L^\infty}  + C e^{s - \tau }s^{\frac{3}{2}}  \left\| \frac{q_{1,-}(.,\tau)}{ 1 + |y|^3} \right\|_{L^\infty} +  \frac{C (1 + s - \tau)e^{s - \tau}}{\sqrt s},
\end{equation}
\begin{equation}\label{outer-Phi-e}
\left\| q_{2,e} (.,s) \right\|_{L^\infty} \leq  C e^{- \frac{(s -\tau)}{p} } \|q_{2,e}(.,\tau)\|_{L^\infty}  + C e^{s - \tau }s^{\frac{3}{2}}  \left\| \frac{q_{2,-}(.,\tau)}{ 1 + |y|^3} \right\|_{L^\infty} +  \frac{C (1 + s - \tau)e^{s - \tau}}{s^{\frac{p_1 +2}{2}}}.
\end{equation}
\end{itemize} 
\end{proposition}
\begin{proof}
The proof of this Proposition  is  given  in two steps:

+ \textit{ Step 1: }  We will give a   proof   to items $(i)$ and $(ii)$ by using the projection the  equations which are  satisfied by $q_1 $ and $  q_2$.

+ \textit{ Step 2: } We will control the other components   by studying   the dynamics  of the linear operator $\mathcal{L}  + V $.

a) \textbf{Step 1: } We observe that the techniques of the   proof  for  \eqref{ODE-q-1-0-1}, \eqref{ODE-Phi-0-1}, \eqref{ODE-q-1-2} and \eqref{ODE-Phi-2} are the  same. So, we only deal  with the proof  of 
\eqref{ODE-q-1-2}. For each $j,k \leq n$ by using the  equation in \eqref{equation-satisfied-by-q-1-2}  and the definition of $q_{1,j,k}$ we deduce  that
\begin{equation}\label{in-ject-equa-q-1-h-j-k}
\left|q_{1,i,j}'(s)  -  \int \left[ \mathcal{L} q_1 + Vq_1   + B_1(q_1, q_2) + R_1(y,s) \right] \chi(y,s)  \left( \frac{y_i y_j }{4} - \frac{\delta_{i,j}}{2} \right) \rho dy  \right|  \leq C e^{-s},
\end{equation}
if $K$ is  large enough.  In addition to that,  using the fact   $(q_1,q_2) \in V_A(s)$   and    Lemma     \ref{lemma-estiam-q-1-2-in V-A},  Lemma \ref{lemmas-potentials}, Lemma  \ref{lemma-quadratic-term-B-1-2}, Lemma   \ref{lemma-rest-term-R-1-2}     that
\begin{eqnarray*}
\left|\int   \mathcal{L}(q)\chi  \left( \frac{y_i y_j }{4} - \frac{\delta_{i,j}}{2} \right)\rho dy  \right| &\leq & \frac{C}{s^3},\\
\left| \int   V q_1 \chi  \left( \frac{y_i y_j }{4} - \frac{\delta_{i,j}}{2} \right) \rho  dy  + \frac{2 }{s} q_{1,i,j}(s)\right| &\leq & \frac{C A}{s^3},\\
 \left|\int B_1(q_1,q_2) \chi  \left( \frac{y_i y_j }{4} - \frac{\delta_{i,j}}{2} \right) \rho dy \right| &\leq & \frac{C }{s^3},\\
\left|  \int  R_1(y,s) \chi  \left( \frac{y_i y_j }{4} - \frac{\delta_{i,j}}{2} \right) \rho  dy \right| &\leq & \frac{C}{s^3},
\end{eqnarray*}
if $s \geq s_4(A)$. Then, \eqref{ODE-q-1-2} is  derived  by adding all the   above   estimates.

\textbf{ Step 2:} In this part,  we will concentrate
 on  the  proof of    items $(iii)$ and $(iv)$. 
 We now rewrite \eqref{equation-satisfied-by-q-1-2} in its integral form: for each $s \geq \tau $
\begin{equation}\label{equationq-1-2-intergral}
\left\{   \begin{array}{rcl}
q_1(s)  &=& \mathcal{K}(s,\tau) q_1(\tau) + \int_{\tau}^s \mathcal{K}(s,\sigma) \left[ (V_{1,1} q_1 )(\sigma) +  ( V_{1,2} q_2) (\sigma) +   B_1(q_1,q_2)(\sigma)  + R_1(\sigma) \right] d \sigma  \\
&=&  \sum_{i=1}^5 \vartheta_{1,i}(s,\tau),\\
q_2(s) & =& \mathcal{K}(s,\tau) q_2(\tau) + \int_{\tau}^s \mathcal{K}(s,\sigma) \left[ (V_{2,1} q_1)(\sigma) +  (V_{2,2} q_2)(\sigma)  +  B_2(q_1,q_2)(\sigma)  + R_2(\sigma) \right] d \sigma \\
&=&    \sum_{i=1}^5 \vartheta_{2,i}(s,\tau).
\end{array} \right.
\end{equation}
where  $\{\mathcal{K}(s,\tau)\}_{s \geq \tau}$    is the fundamental  solution  associated to the linear operator  $\mathcal{L} + V$  and  defined by 
\begin{equation}\label{fundamental-sol}
\left\{ \begin{array}{l}
\partial_s \mathcal{K}(s,\tau)  = (\mathcal{L} + V)  \mathcal{K}(s,\tau),\quad  \forall   s > \tau,\\
\mathcal{K}(\tau,\tau) = Id. 
\end{array}
\right.
\end{equation}
Let us now introduce some notations:
\begin{align*}
\vartheta_{1,1}(s,\tau) &= \mathcal{K}(s,\tau) q_1(\tau), \quad \vartheta_{1,2}(s,\tau) = \int_{\tau}^s \mathcal{K}(s,\sigma) (V_{1,1} q_1)(\sigma)  d \sigma, \quad  \vartheta_{1,3}(s,\tau) = \int_{\tau}^s \mathcal{K}(s,\sigma)  ( V_{1,2} q_2)(\sigma)   d \sigma,\\
 \vartheta_{1,4}(s,\tau)   & = \int_{\tau}^s \mathcal{K}(s,\sigma)  ( B_1 (q_1,q_2))(\sigma)   d \sigma, \quad  \vartheta_{1,5} =  \int_{\tau}^s \mathcal{K}(s,\sigma)  ( R_1 (., \sigma)) d \sigma,
\end{align*}
and 
\begin{align*}
\vartheta_{2,1}(s,\tau) &= \mathcal{K}(s,\tau) (q_2(\tau)), \quad \vartheta_{2,2}(s,\tau) = \int_{\tau}^s \mathcal{K}(s,\sigma) (V_{2,1} q_1)(\sigma)  d \sigma, \quad  \vartheta_{2,3}(s,\tau) = \int_{\tau}^s \mathcal{K}(s,\sigma)  ( V_{2,2} q_2)(\sigma)   d \sigma,\\
 \vartheta_{2,4}(s,\tau)   & = \int_{\tau}^s \mathcal{K}(s,\sigma)  ( B_2 (q_1,q_2))(\sigma)   d \sigma, \quad  \vartheta_{2,5}  = \int_{\tau}^s \mathcal{K}(s,\sigma)  ( R_2 (., \sigma)) d \sigma.
\end{align*}
From \eqref{equationq-1-2-intergral}, we can see   the strong influence of the  kernel  $\mathcal{K}.$  For that   reason, we will study the dynamics of that operator:    
\begin{lemma}[A priori estimates of the linearized operator]\label{dynamic-K-feym}For all $\rho^* \geq 0$, there exists $s_5(\rho^*)  \geq 1$, such that if $\sigma \geq s_5(\rho^*)$ and $v \in L^{2}_{\rho}$ satisfies 
\begin{equation}\label{condition-q-sigma}
\sum_{m=0}^2 |v_m| + \left\|\frac{v_-}{1 + |y|^3}\right\|_{L^{\infty}}   +\|v_e\|_{L^{\infty}} < \infty,
\end{equation} 
then, for all $ s \in [\sigma, \sigma + \rho^*],$  the function $\theta(s) = \mathcal{K}(s,\sigma) v $ satisfies 
\begin{equation}\label{control-K-q-sigma1}
\begin{array}{l}
\left\|\frac{\theta_-(y,s)}{1 + |y|^3}\right\|_{L^{\infty}} \leq \frac{C e^{s -\sigma} \left( (s - \sigma)^2   +1 \right)}{s} \left(  |v_0|  + |v_1|  + \sqrt s |v_2|\right)\\
\quad \quad \quad \quad \quad + C e^{-\frac{(s-\sigma)}{2}} \left\|\frac{v_-}{1 + |y|^3}\right\|_{L^{\infty}} + C \frac{e^{-(s-\sigma)^2}  }{s^{\frac{3}{2}}} \|v_e\|_{L^{\infty}},
\end{array}
\end{equation}
and
\begin{equation}\label{control-K-q-e}
\|\theta_e(y,s)\|_{L^{\infty}} \leq  C e^{s -\sigma} \left(  \sum_{l=0}^2 s^{\frac{l}{2}} |v_l|    +s^{\frac{3}{2}} \left\|\frac{v_-}{1  +|y|^3}\right\|_{L^{\infty}}\right) + C e^{-\frac{s -\sigma}{p}} \|v_e\|_{L^{\infty}}.
\end{equation}
\end{lemma}
\begin{proof} The proof of this result was given by Bricmont and Kupiainen \cite{BKnon94} in the one dimensional case.  Later, it was  extended to the  higher dimensional case  by Nguyen and Zaag  \cite{NZens16}. We kindly refer interested readers to Lemma 2.9 in \cite{NZens16} for  details of the proof.
\end{proof}

We now use Lemmas  \ref{dynamic-K-feym}, \ref{lemma-estiam-q-1-2-in V-A}, \ref{lemmas-potentials},   \ref{lemma-quadratic-term-B-1-2} and \ref{lemma-rest-term-R-1-2}   to deduce the following Lemma which implies Proposition \ref{prop-dynamic-q-1-2-alpha-beta}.
\begin{lemma}\label{control-prin-q-e-q-}
For all $A \geq 1, \rho^* \geq 0$,  there exists $s_6(A, \rho^*) \geq 1$ such that $\forall s_0 \geq s_6(A,\rho^*)$ and $q(s) \in S_A(s), \forall s \in [\tau, \tau + \rho^*] \text{ where } \tau \geq s_0$. Then, we have the following properties: for all $ s \in [\tau, \tau + \rho^*] $,
\begin{itemize}
\item[$i)$] (The linear term $\vartheta_{1,1}(s,\tau)$ and $\vartheta_{2,1}(s,\tau)$)
\begin{eqnarray*}
\left\|\frac{(\vartheta_{1,1}(s,\tau))_-}{1 + |y|^3} \right\|_{L^{\infty}} &\leq & C e^{- \frac{s - \tau}{2}}  \left\|\frac{q_{1,-}(., \tau)}{1 + |y|^3} \right\|_{L^{\infty}}  +  \frac{C e^{- (s-\tau)^2}}{s^{\frac{3}{2}}} \|q_{1,e} (\tau)\|_{L^\infty}   + \frac{C }{s^{2}} ,\\
\| (\vartheta_{1,1}(s,\tau))_e\|_{L^{\infty}} &\leq & C e^{- \frac{s - \tau }{p}} \|q_{1,e} (\tau)\|_{L^\infty}  +   C e^{s- \tau} s^{\frac{3}{2}} \left\|\frac{q_{1,-}(., \tau)}{1 + |y|^3} \right\|_{L^{\infty}}  + \frac{C }{ \sqrt s},\\
\left\|\frac{(\vartheta_{2,1}(s,\tau))_-}{1 + |y|^3} \right\|_{L^{\infty}} &\leq & C e^{- \frac{s - \tau}{2}}  \left\|\frac{q_{2,-}(., \tau)}{1 + |y|^3} \right\|_{L^{\infty}}  +  \frac{C e^{- (s-\tau)^2}}{s^{\frac{3}{2}}} \|q_{2,e} (\tau)\|   + \frac{C }{s^{\frac{p_1+ 5}{2}}} ,\\
\| (\vartheta_{2,1}(s,\tau))_e\|_{L^{\infty}} &\leq & C e^{- \frac{s - \tau }{p}} \|q_{2,e} (\tau)\|_{L^\infty}  +   C e^{s- \tau} s^{\frac{3}{2}} \left\|\frac{q_{2,-}(., \tau)}{1 + |y|^3} \right\|_{L^{\infty}}  + \frac{C }{ s^{\frac{p_1 + 2}{2}}}.
\end{eqnarray*}
\item[$ii)$] The quadratic term $\vartheta_{1,2}(s,\tau)$ and $\vartheta_{2,2}(s,\tau)$
\begin{eqnarray*}
\left\|\frac{(\vartheta_{1,2}(s,\tau))_-}{1 + |y|^3} \right\|_{L^{\infty}} &\leq & \frac{C(s - \tau)}{ s^{2 }}, \quad \| (\vartheta_{1,2}(s,\tau))_e\|_{L^{\infty}} \leq  \frac{C (s - \tau)}{ s^{\frac{1}{2}  }},\\
\left\|\frac{(\vartheta_{2,2}(s,\tau))_-}{1 + |y|^3} \right\|_{L^{\infty}} &\leq & \frac{C(s - \tau)}{ s^{ \frac{p_1 + 5}{2} }}, \quad \| (\vartheta_{2,2}(s,\tau))_e\|_{L^{\infty}} \leq  \frac{C (s - \tau)}{ s^{\frac{p_1 + 2}{2}  }}.
\end{eqnarray*}
 \item[$iii)$] The correction terms $\vartheta_{1,3}(s,\tau) $ and  $\vartheta_{2,3}(s,\tau) $ 
\begin{eqnarray*}
\left\|\frac{(\vartheta_{1,3}(s,\tau))_-}{1 + |y|^3} \right\|_{L^{\infty}} &\leq & \frac{C (s  - \tau) }{ s^{2}}, \quad \| (\vartheta_{1,3}(s,\tau))_e\|_{L^{\infty}} \leq  \frac{C (s  -\tau ) }{ s^{\frac{1}{2} }},\\
\left\|\frac{(\vartheta_{2,3}(s,\tau))_-}{1 + |y|^3} \right\|_{L^{\infty}} &\leq & \frac{C (s  - \tau) }{ s^{\frac{p_1+5}{2}}}, \quad \| (\vartheta_{2,3}(s,\tau))_e\|_{L^{\infty}} \leq  \frac{C (s  -\tau ) }{ s^{\frac{p_1+2}{2} }}.
\end{eqnarray*}
\item[$iv)$] The correction terms $\vartheta_{1,4}(s,\tau) $ and  $\vartheta_{2,4}(s,\tau) $ 
\begin{eqnarray*}
\left\|\frac{(\vartheta_{1,3}(s,\tau))_-}{1 + |y|^3} \right\|_{L^{\infty}} &\leq & \frac{C (s  - \tau) }{ s^{2}}, \quad \| (\vartheta_{1,3}(s,\tau))_e\|_{L^{\infty}} \leq  \frac{C (s  -\tau ) }{ s^{\frac{1}{2} }},\\
\left\|\frac{(\vartheta_{2,3}(s,\tau))_-}{1 + |y|^3} \right\|_{L^{\infty}} &\leq & \frac{C (s  - \tau) }{ s^{\frac{p_1+5}{2}}}, \quad \| (\vartheta_{2,3}(s,\tau))_e\|_{L^{\infty}} \leq  \frac{C (s  -\tau ) }{ s^{\frac{p_1+2}{2} }}.
\end{eqnarray*}
\item[$v)$] The correction terms $\vartheta_{1,5}(s,\tau) $ and  $\vartheta_{2,5}(s,\tau) $ 
\begin{eqnarray*}
\left\|\frac{(\vartheta_{1,3}(s,\tau))_-}{1 + |y|^3} \right\|_{L^{\infty}} &\leq & \frac{C (s  - \tau) }{ s^{2}}, \quad \| (\vartheta_{1,3}(s,\tau))_e\|_{L^{\infty}} \leq  \frac{C (s  -\tau ) }{ s^{\frac{1}{2} }},\\
\left\|\frac{(\vartheta_{2,3}(s,\tau))_-}{1 + |y|^3} \right\|_{L^{\infty}} &\leq & \frac{C (s  - \tau) }{ s^{\frac{p_1+5}{2}}}, \quad \| (\vartheta_{2,3}(s,\tau))_e\|_{L^{\infty}} \leq  \frac{C (s  -\tau ) }{ s^{\frac{p_1+2}{2} }}.
\end{eqnarray*}
\end{itemize}
\end{lemma} 
\begin{proof} 
 The result is implied from  the definition of  the shrinking set $V_A(s) $ and  Lemma  \ref{lemma-estiam-q-1-2-in V-A}   and  the bounds for $V, V_{j,k},  B_1, B_2, R_1, R_2$  with $j ,k  \in \{1,2\}$  which are shown in Lemmas   \ref{lemmas-potentials},   \ref{lemma-quadratic-term-B-1-2} and \ref{lemma-rest-term-R-1-2}.  For details  in a  quite  similar case,  see Lemma 4.20 in Tayachi and Zaag \cite{TZpre15}.
\end{proof}
Finally,    the  conclusion   of $(iii)$ and $(iv)$ of Proposition \ref{prop-dynamic-q-1-2-alpha-beta} follows by  using   formular   \eqref{equationq-1-2-intergral}  and    Lemma \eqref{control-prin-q-e-q-}. This concludes  the proof of  Proposition \ref{prop-dynamic-q-1-2-alpha-beta}.
\end{proof}

\subsection{ Conclusion of the proof of  Proposition \ref{pro-reduction- to-finit-dimensional}}
In this subsection, we will give prove a Proposition which implies Proposition  \ref{pro-reduction- to-finit-dimensional} directly. More precisely,  this is our statement:
\begin{proposition}\label{control-q(s)-V-A-s-1-2}
There exists $A_7\geq 1$ such that for all  $ A \geq A_7$, there exists $s_7(A)\geq 1$ such that for all $s_0 \geq s_7(A)$, we have the following properties: If the following conditions hold:
\begin{itemize}
\item[$a)$] $(q_1,q_2)(s_0) = (\phi_1,\phi_2)  $ with   $(d_0,d_1) \in \mathcal{D}_{A,s_0}$,
\item[$b)$] For all $s \in [s_0,s_1]$ we have  $(q_1,q_2)(s) \in V_A(s)$.
\end{itemize}
Then for all $s \in [s_0,s_1]$, we have
\begin{eqnarray}
 \forall i,j \in \{1, \cdots, n\}, \quad |q_{2,i,j}(s)| & \leq & \frac{A^2 \ln s}{2 s^2}, \label{conq_1-2} \\
\left\| \frac{q_{1,-}(y,s)}{1 + |y|^3}\right\|_{L^{\infty}} &\leq &  \frac{A}{2 s^{2}}, \quad \|q_{1,e}(s)\|_{L^{\infty}} \leq  \frac{A^2}{2 \sqrt s}, \label{conq-q-1--and-e}\\
\left\| \frac{q_{2,-}(y,s)}{1 + |y|^3}\right\|_{L^{\infty}}  &\leq &  \frac{A^2}{2 s^{\frac{p_1 + 5}{2}}},\quad \|q_{2,e}(s)\|_{L^{\infty}} \leq  \frac{A^3}{2 s^{\frac{p_1 + 2}{2}}}. \label{conq-q-2--and-e}
\end{eqnarray}
where $\mathcal{D}_{A,s_0}$ is introduced in Lemma \ref{lemma-control-initial-data} and $(\phi_1,\phi_2)$ is defined as in Definition \eqref{initial-data-profile-complex}.
\end{proposition}
\begin{proof}
The proof relies  on  Propostion  \ref{prop-dynamic-q-1-2-alpha-beta} and  details are similar to   Proposition 4.7 of Merle and Zaag \cite{MZdm97}. For that reason, we only give a short proof to \eqref{conq_1-2}. We use \eqref{ODE-q-1-2} to deduce that  
$$ \left|  \int_{s_0}^s (\tau^2q_{j,k} (\tau))d \tau   \right|  \leq CA ( \ln (s) -   \ln (s_0)),  $$
which yields 
$$  |q_{1,j,k} (s)| \leq C A s^{-2}\ln s \leq \frac{A^2 \ln s}{2 s^2 },$$
if $A \geq  A_7$ large enough and $s \geq s_7 (A)$. Then,    \eqref{conq_1-2}  follows.
\end{proof}
We here give the conclusion of the  proof of    Proposition  \ref{pro-reduction- to-finit-dimensional}:
\begin{proof}
From Proposition  \ref{control-q(s)-V-A-s-1-2},  if $(q_1,q_2)(s_1)  \in  \partial V_A(s_1)$  then: 
\begin{equation}\label{q-1-0-j-k-in-partial-hat-V-A-s-1}
\left(q_{1,0}, (q_{1,j})_{1 \leq  j\leq n}, q_{2,0},(q_{2,j})_{1 \leq  j\leq n},(q_{2,j,k})_{1 \leq j,k \leq n}\right)(s_1) \in  \partial \hat V_A (s_1). 
\end{equation}
This concludes item $(i)$ of   Proposition \ref{pro-reduction- to-finit-dimensional}.

\noindent
\text{The proof of item $(ii)$ of  Proposition \ref{pro-reduction- to-finit-dimensional}.}  Thanks to  \eqref{q-1-0-j-k-in-partial-hat-V-A-s-1}, we derive two the following cases:

+ The first case: There  exists $j_0  \in \{ 1,...,n\}$ and $\epsilon_0 \in \{-1, 1\}$ such that either  $q_{1,0} (s_1) = \epsilon_0 \frac{A}{s_1^2}$ or $ q_{1,j_0} = \epsilon_0 \frac{A}{s_1^2}$ or $q_{2,0} = \epsilon_0\frac{A^2}{s_1^{p_1+ 2}}$ or $q_{2,j_0} (s_1) = \epsilon_0 \frac{A^2}{s_1^{p_1 + 2}}$. Without loss of generality, we can suppose that $ q_{1,0} = \epsilon_0 \frac{A}{s_1^2} $ (the other cases  are similar). Then,  by using    \eqref{ODE-q-1-0-1}, we can prove that the sign of $q_{1,0}'(s_1)$  is oppsite to the sign of $\left( \epsilon_0 \frac{A}{s_1^2}\right)'$. In other words,
 $$ \epsilon_0  \left( q_{1,0} - \epsilon_0 \frac{A}{s^{2}}  \right)'(s_1) > 0. $$

+ The second case: There exists $j_0, k_0 , \epsilon_0 \in {-1,1}$  such that $q_{2,j_0,k_0} (s_1) = \epsilon_0 \frac{A^2}{s_1^{p_1+2}}$, by using  \eqref{ODE-Phi-2} we can prove that
$$ \epsilon_0  \left( q_{2,j_0,k_0} - \epsilon_0 \frac{A^2}{s^{p_1 + 2}}  \right)'(s_1) > 0. $$

Finally,  we deduce that there exists $\delta_0 > 0$ such that for all $\delta \in (0, \delta_0)$ we have
$$\left(q_{1,0}, (q_{1,j})_{1 \leq  j\leq n}, q_{2,0},(q_{2,j})_{1 \leq  j\leq n},(q_{2,j,k})_{1 \leq j,k \leq n}\right)(s_1 + \delta ) \notin   \hat V_A (s_1 + \delta ). $$
if $A \geq A_3$ and $s_0 \geq s_3(A)$ large enough. Then, the item $(ii)$  of Proposition    follows. Hence, we also derive the conclusion of Proposition \ref{pro-reduction- to-finit-dimensional}.
\end{proof}

\appendix
 \section{Appendix}
 In this appendix, we  state and prove several  technical  and  and straightforward results need in our paper.
 
 \medskip
 We first give a Taylor expansion of the  quadratic terms defined in  \eqref{defini-bar-B-1} and  \eqref{defini-bar-B-2}.
\begin{lemma}[Asymptotics  of $\bar B_1$ and $ \bar B_2 $]\label{asymptotic-bar-B-1-2}
We consider  $\bar B_1 (\bar w_1, w_2)$  and $  \bar B_2 ( \bar w_1, w_2)$   as defined  in \eqref{defini-bar-B-1} and  \eqref{defini-bar-B-2}. Then, the following holds
 \begin{eqnarray}
\bar B_1 (\bar w_1, w_2) &=& \frac{p}{2 \kappa} \bar w^2_1  + O (|\bar w_1|^3 + |w_2|^2)\label{asymptotic-bar-B-1},\\
\bar B_2( \bar w_1, w_2) &=& \frac{p}{ \kappa} \bar w_1 w_2  + O \left( |\bar w_1|^2 |w_2| \right) + O  \left( |w_2|^3 \right),\label{asymptotic-bar-B-2}                                      
\end{eqnarray}
as $(\bar w_1, w_2) \to (0,0)$.
\end{lemma} 
\begin{proof}    Using the  Newton  binomial formula (remember that  $p \in \mathbb{N}$), we derive that:     
$$ ( \bar w_1  + \kappa + i w_2)^p = ( \bar w_1   +  \kappa)^p + i p ( \bar w_1  + \kappa)^{p - 1}  w_2  +  p (p - 1)  ( \bar w_1  + \kappa)^{p - 2}  w_2^2    +    G (\bar w_1, w_2), $$
with 
$$ \left| G (\bar w_1, w_2)  \right|  \leq C | w_2|^3, \quad \forall    |\bar w_1|  + |w_2| \leq 1.$$
Then,  
\begin{eqnarray}
\text{  Re } \left(  (\bar w_1  + \kappa + i w_2  )^{p}  \right)  &=&  ( \bar w_1   +  \kappa)^p  +   p (p - 1)  ( \bar w_1  + \kappa)^{p - 2}  w_2^2  +  \text{ Re } (G) ,\label{real-part-w-1-w_--2-bar}\\
\text{  Im } \left(  (\bar w_1  + \kappa + i w_2  )^{p}  \right)  &=&   p (\bar w_1 + \kappa)^{p-1} w_2 +    \text{ Im } (G). \label{imaginary-part-w-1-w_--2-bar}
\end{eqnarray}
Moreover,  we apply  again  the Newton  binomial formula  to $ (\kappa   +  \bar w_1)^{p}, (\kappa   +  \bar w_1)^{p-1}$   around  $\bar w_1   = 0$ and we get 
\begin{eqnarray}
(\kappa   +  \bar w_1)^{p}   & =  &     \kappa^p   + \frac{p}{p-1}  \bar w_1  + \frac{p}{ 2 \kappa} \bar w_1^2 +  O (|\bar w_1 |^3), \label{expansion-Taylor-kappa+w_1-p}\\
 (\kappa   +  \bar w_1)^{p - 1}   & = &     \frac{1}{ p-1} +  \frac{1}{\kappa} \bar w_1 + O ( |\bar w_1|^2)\label{expansion-Taylor-kappa+w_1-p-1} .
\end{eqnarray}
  Then,   \eqref{asymptotic-bar-B-1}   follows by    \eqref{real-part-w-1-w_--2-bar} and  \eqref{expansion-Taylor-kappa+w_1-p}   and   \eqref{asymptotic-bar-B-2}  follows by  \eqref{imaginary-part-w-1-w_--2-bar} and \eqref{expansion-Taylor-kappa+w_1-p-1}.
\end{proof}
Now, we give an expansion of the potentials defined in \eqref{defini-potentian-V}  and  \eqref{defini-V-1-1} - \eqref{defini-V-2-2}. The following is our statement:
\begin{lemma}[The potential functions $V$ and  $ V_{j,k} $ with $ j,k \in \{1,n\}$]\label{lemmas-potentials}  We consider  $ V , V_{1,1},V_{1,2},V_{2,1} $ and $ V_{2,2}$ as  defined  in \eqref{defini-potentian-V}  and  \eqref{defini-V-1-1} - \eqref{defini-V-2-2}. Then, the following holds:

$(i)$ For all $s \geq 1 $ and  $y \in \mathbb{R}^n$,  we have  $| V(y,s)| \leq C,$ 
   \begin{equation}\label{esti-y-R-n-V-1-1-and-22}
 \left| V (y,s)\right|  \leq \frac{C (1 + |y|^2)}{s},
 \end{equation}
 and 
 \begin{equation}\label{the-potential-V-inside-blowup-region}
	V(y,s) = - \frac{(|y|^2 -2 n  )}{4 s} + \tilde V(y,s),
\end{equation}
where $\tilde V$ satisfies 
\begin{equation}\label{defini-tilde V.}
|\tilde V (y,s)| \leq C \frac{(1 + |y|^4)}{s^2}, \forall s \geq 1, |y| \leq 2 K \sqrt s. 
\end{equation}
 
 $(ii)$ For all $s \geq 1$ and  $y \in \mathbb{R}^n,$ the  potential functions  $V_{j,k}$ with $j,k \in \{ 1,2\}$ satisfy
 \begin{eqnarray*}
 \|V_{1,1} \|_{L^\infty}  + \| V_{2,2}\|_{L^\infty}  &\leq &  \frac{C}{s^2},\\
  \|V_{1,2} \|_{L^\infty}  + \| V_{2,1}\|_{L^\infty}  &\leq &  \frac{C}{s},\\
 \left|  V_{1,1} (y,s) \right|  + \left| V_{2,2} (y,s) \right| &\leq & \frac{C (1 + |y|^4)}{s^4},\\ 
 \left|  V_{1,2} (y,s) \right|  + \left| V_{2,1} (y,s) \right| &\leq & \frac{C (1 + |y|^2)}{s^2}.
 \end{eqnarray*}
\end{lemma}
\begin{proof}  We see that  item $(ii)$ is derived  directly from the defintion of $V_{j,k}$.   In addition to that,  the proof of    $(i)$ is quite similar to  Lemma B.1,  page 1270   in \cite{NZens16}. 
\end{proof}

Now, we give a Taylor expansion of the quadratics  terms  $B_1$ and $B_2 $ given in \eqref{defini-quadratic-B-1} and   \eqref{defini-term-under-linear-B-2} .
\begin{lemma}[The quadratic terms $B_1 (q_1,q_2)$ and $B_2 (q_1,q_2)$]\label{lemma-quadratic-term-B-1-2} We consider $B_1 (q_1,q_2) $ and $B_2 (q_1,q_2)$  as defined  in  \eqref{defini-quadratic-B-1} and   \eqref{defini-term-under-linear-B-2} respectively. For all $A \geq 1,$ there exists $s_8 (A) \geq 1$ such that for all $s \geq s_{8} (A),$  if $(q_1,q_2) (s) \in V_A(s),$ then 
\begin{eqnarray}
\left|  B_1 (q_1,q_2)   \right| &\leq &  C \left( |q_1|^2 + |q_2|^2 \right),\label{estimate-B-1-q-1-2-inside-blowup}\\
\left|  B_2 (q_1, q_2) \right| &\leq & C \left(   \frac{|q_1|^2}{s} + |q_1 . q_2|   + |q_2|^2\right).   \label{estimate-B-2-q-1-2-inside-blowup}
\end{eqnarray}
\end{lemma}
\begin{proof}
We  first  recall the  two functions   $F_1 (u_1,u_2)$ and $ F_2(u_1,u_2)$ which  are  defined  in  \eqref{defi-mathbb-A-1-2}.   As a matter of facts,  they  belong to $C^\infty(\mathbb{R}^2)$. Then, by applying  a  Taylor expansion to $F_1,F_2,$ we obtain
\begin{eqnarray*}
F_1 (\Phi_1  + q_1, \Phi_2 + q_2)  &  =  &  \sum_{j,k \leq p} \frac{1}{j! k!}\partial_{u_1^ju_2^k}^{j+k} F_1 (\Phi_1, \Phi_2) q_1^j q_2^k, \\
F_2 \left(\Phi_1  + q_1,  \Phi_2 + q_2 \right) & = &\sum_{j,k \leq p} \frac{1}{j! k!}\partial_{u_1^ju_2^k}^{j+k} F_2 (\Phi_1, \Phi_2) q_1^j q_2^k.
\end{eqnarray*}
Then,  \eqref{estimate-B-1-q-1-2-inside-blowup} and \eqref{estimate-B-2-q-1-2-inside-blowup} follow  by  definition of  $B_1, B_2$ and also the definition of the shrinking set $V_A(s)$.
\end{proof}
In the following lemma, we give  various estimates involing the rest terms $R_1$ and $R_2$ defined  in  \eqref{defini-the-rest-term-R-1} and \eqref{defini-the-rest-term-R-2}.
\begin{lemma}[The rest terms $R_1, R_2$]\label{lemma-rest-term-R-1-2} 
For all $s \geq 1,$ we consider  $R_1, R_2$  defined  in \eqref{defini-the-rest-term-R-1} and \eqref{defini-the-rest-term-R-2}. Then,  
\begin{itemize}
\item[$(i)$]  For all $s \geq 1$ and $y \in \mathbb{R}^n$
\begin{eqnarray*}
R_1 (y,s) &= & \frac{c_{1,p}}{ s^2}   +  \tilde R_1 (y,s),\\
 R_2 (y,s)  &=&  \frac{c_{2,p}}{s^3}  + \tilde R_2 (y,s),
\end{eqnarray*}
where $c_{1,p} $and  $ c_{2,p}$ are constants depended on $p$ and   $\tilde  R_1, \tilde  R_2$ satisfy: for all $|y| \leq 2 K\sqrt{s}$ 
\begin{eqnarray*}
|\tilde R_1 (y,s) | & \leq &  \frac{C (1 + |y|^4)}{s^3},\\
| \tilde R_2 (y,s) | & \leq &  \frac{C (1 + |y|^6)}{s^4}.
\end{eqnarray*}
\item[$(ii)$] Moreover, we have for all $s \geq 1$
\begin{eqnarray*}
\| R_1(.,s)\|_{L^\infty(\mathbb{R}^n)}  & \leq  & \frac{C}{s},\\
\| R_2(.,s)\|_{L^\infty(\mathbb{R}^n)}  & \leq & \frac{C}{s^2},
\end{eqnarray*}
\end{itemize}
\end{lemma}
\begin{proof}
   The proofs for $R_1$ and $R_2 $ are quite similar.  For that  reason, we only give the proof of  the estimates   on $R_2$.  This  means that  we need to  prove the following estimates:
\begin{equation}\label{estimate-R-2-inside}
R_2 (y,s)  = -  \frac{n(n+4) \kappa}{(p-1)s^3}  + \tilde R_2 (y,s),
\end{equation}
with 
$$ | \tilde R_2 (y,s) |  \leq    \frac{C (1 + |y|^6)}{s^4}, \forall |y| \leq 2 K \sqrt s$$
and  
\begin{equation}\label{estima-all-spa-R-2}
\| R_2(.,s)\|_{L^\infty}  \leq \frac{C}{s^2}.
\end{equation}
We first  from \eqref{defini-the-rest-term-R-2}, recall the definition of $R_2(y,s)$
$$  R_2 (y,s) = \Delta \Phi_2 - \frac{1}{2} y \cdot \nabla \Phi_2  - \frac{\Phi_2}{p-1} + F_2 (\Phi_1, \Phi_2) - \partial_s \Phi_2, $$
Then, we can rewrite  $R_2$ as follows
$$R_2 (y,s) =  \Delta \Phi_2 - \frac{1}{2} y \cdot \nabla \Phi_2  - \frac{\Phi_2}{p-1} + p \Phi_1^{p-1} \Phi_2 - \partial_s \Phi_2  + R_2^* (y,s),$$
where 
$$  R_2^*(y,s) = F_2 (\Phi_1, \Phi_2)  - p \Phi_1^{p-1} \Phi_2.$$
 Using  the definition  of $F_2$ in \eqref{defi-mathbb-A-1-2},   and the defintions of  $\Phi_1, \Phi_2$ in  \eqref{defi-Phi-1} and \eqref{defi-Phi-2},  we derive   that 
\begin{eqnarray*}
|R_2^* (y,s) | & \leq  &\frac{C ( 1 + | y|^6)}{s^3}, \quad \forall  |y| \leq 2 K \sqrt s,\\
& \text{ and } &\\
\| R_2^*(y,s)\|_{L^\infty } & \leq  &\frac{C}{s^2}.
\end{eqnarray*} 
In addition to that, we introduce $\tilde R_2$ as follows: 
$$ \bar R_2 (y,s) =   \Delta \Phi_2 - \frac{1}{2} y \cdot \nabla \Phi_2  - \frac{\Phi_2}{p-1} + p \Phi_1^{p-1} \Phi_2 - \partial_s \Phi_2.$$
Then,   we  may  obtain the conclusion if  the following  two estimates  hold:
\begin{eqnarray}
 \left| \bar  R_2 (y,s)  + \frac{n(n+4)\kappa }{(p-1)s^3}\right|  &\leq & \frac{C (1 + |y|^6)}{s^4}, \label{estimates-bar-R-2-y-2-K-s}\\
 \| \bar R_2 (.,s)\|_{L^{\infty} (\mathbb{R}^n)} & \leq  & \frac{C}{s^2}.\label{estimates-all-sapce-bar-R-2}
\end{eqnarray}
\noindent \textit{+ The proof  of \eqref{estimates-bar-R-2-y-2-K-s}:}
We first aim at  expanding  $\Delta \Phi_2$ in  a polynomial in  $y$  of order less than $4$ via  the  Taylor expansion. Indeed,   $\Delta \Phi_2$ is given by
\begin{eqnarray*}
 \Delta \Phi_2 &=&  \frac{2 n }{s^2} \left( p-1 + \frac{(p-1)^2 |y|^2}{4p s} \right)^{-\frac{p}{p-1}} -  \frac{(p-1)|y|^2}{s^3} \left( p-1 + \frac{(p-1)^2}{4 p} \frac{|y|^2}{s}\right)^{- \frac{2p-1}{p-1}} \\
 & -& \frac{(n+ 2) (p-1)|y|^2}{2 s^3} \left( p-1 + \frac{(p-1)^2}{4 p} \frac{|y|^2}{s}\right)^{- \frac{2p-1}{p-1}}  + \frac{(2p-1)(p-1)^2 |y|^4}{4 p s^4} \left( p-1 + \frac{(p-1)^2}{4 p} \frac{|y|^2}{s}\right)^{- \frac{3p-2}{p-1}}. 
\end{eqnarray*}
Besides that,   we  make a   Taylor expansion  in the  variable  $z  =  \frac{|y|}{\sqrt s}$  for   $\left(  p-1  + \frac{(p-1)^2 }{4 p} \frac{|y|^2}{s} \right)^{- \frac{p}{p-1}}$   when   $|z| \leq 2 K$, and we get
$$  \left| \left( p-1 + \frac{(p-1)^2 |y|^2}{4p s} \right)^{-\frac{p}{p-1}}  - \frac{\kappa}{p-1}+ \frac{\kappa}{4 (p-1) } \frac{|y|^2}{s} \right|    \leq  \frac{C (1 + |y|^4)}{s^2}, \forall |y| \leq 2 K \sqrt s.$$
which yields 
$$  \left|  \frac{2 n }{s^2} \left( p-1 + \frac{(p-1)^2 |y|^2}{4p s} \right)^{-\frac{p}{p-1}} - \frac{2 n \kappa}{ (p-1)s^2}  +  \frac{n \kappa |y|^2}{2 (p-1)s^3} \right| \leq     \frac{C (1 + |y|^4)}{s^4} \leq  \frac{C (1 + |y|^6)}{s^4},  \quad  \forall |y|  \leq 2 K \sqrt{s}.$$
It is similar to   estimate   the other termes     in $\Delta \Phi_2$  as  the above. Finally, we obtain
\begin{equation}\label{Taylor-expan-Delta-Phi_2}
\left|   \Delta \Phi_2  -  \frac{2 n \kappa}{(p-1)s^2} +  \frac{n \kappa |y|^2}{ (p-1)s^3} +  2 \frac{k|y|^2}{(p-1)s^3} \right|   \leq   \frac{C ( 1 + |y|^6) }{s^4} , \forall |y| \leq 2 K \sqrt s.
\end{equation}
As we did  for $\Delta \Phi_2$,  we    estimate  similarly  the   other termes in $\bar R_2$: for all $|y| \leq 2 K \sqrt s$ 
\begin{eqnarray}
\left| - \frac{1}{2} y \cdot \nabla \Phi_2   +  \frac{\kappa |y|^2}{(p-1)s^2}  -  \frac{\kappa |y|^4}{4 (p-1) s^3} -  \frac{\kappa |y|^4}{4 (p-1)s^3} \right|  & \leq &  \frac{C( 1 + |y|^6)}{s^4}, \label{taylor-expan-1-2y-nabla-Phi-2}\\
\left| - \frac{\Phi_2}{p-1} + \frac{\kappa  |y|^2}{(p-1)^2 s^2}  -   \frac{\kappa |y|^4 }{ 4 (p-1)^2 s^3} -  \frac{2 n \kappa }{(p-1)^2 s^2} \right|  & \leq  &\frac{C( 1 + |y|^6)}{s^4}, \label{taylor-expan-1-p-1-Phi_2}\\
\left| p \Phi_1^{p-1}\Phi_2  -  \frac{p \kappa |y|^2}{ (p-1)^2 s^2}  + \frac{(2p -1) \kappa |y|^4}{ 4 (p-1)^2s^3} -  \frac{n \kappa |y|^2}{(p-1)s^3} + \frac{2 p n \kappa}{(p-1)^2 s^2} + \frac{n^2 \kappa}{(p-1)s^3}  \right|& \leq &  \frac{C( 1 + |y|^6)}{s^4},  \label{taylor-expan-1-p-Phi-1-p-1-Phi-2}\\
\left| - \partial_s \Phi_2  -  \frac{2 \kappa |y|^2}{(p-1)s^3} +  \frac{4 n \kappa }{ (p-1)s^3}  \right| &\leq &   \frac{C ( 1 + |y|^6) }{s^4}. \label{taylor-expansion-partial-s-Phi-2}
\end{eqnarray}
Thus,  we use  \eqref{Taylor-expan-Delta-Phi_2}, \eqref{taylor-expan-1-2y-nabla-Phi-2}, \eqref{taylor-expan-1-p-1-Phi_2}, \eqref{taylor-expan-1-p-Phi-1-p-1-Phi-2} and  \eqref{taylor-expansion-partial-s-Phi-2}  to deduce the following
$$ \left| \bar R_2 (y,s)  +  \frac{n(n+4) \kappa}{(p-1)s^3}  \right| \leq  \frac{C(1 + |y|^6)}{s^4},  \quad \forall  |y| \leq 2 K \sqrt s,$$
and    \eqref{estimates-bar-R-2-y-2-K-s} follows
 
\noindent \textit{+ The proof \eqref{estimates-all-sapce-bar-R-2}:}   We rewrite  $\Phi_1, \Phi_2$ as follows
 $$  \Phi_1 (y,s) = R_{1,0} (z)  + \frac{n \kappa}{2 p s}  \text{ and  } \Phi_2 (y,s)  = \frac{1}{s} R_{2,1} (z) - \frac{2n \kappa }{(p-1)s^2} \text{ where } z = \frac{y}{\sqrt s}, $$ 
 where   $R_{1,0}$ and  $R_{2,1}$ are defined  in \eqref{solu-R-0} and \eqref{solu-varphi_1},  respectively. In addition to that, we rewrite $\bar R_2$ in termes of  $R_{1,0}$ and  $R_{2,1}$, and we note that  $R_{1,0}$ and  $R_{2,1}$ satisfy \eqref{equa-R-1-0} and  \eqref{equa-R-2-1}.   Then, it follows that
 $$ |\bar R_2 (y,s)| \leq \frac{C}{s^2}, \forall y \in \mathbb{R}^n.$$ 
 Hence, \eqref{estimates-all-sapce-bar-R-2} follows. This  concludes the proof of  this Lemma.
\end{proof}
 \bibliographystyle{alpha}
\bibliography{mybib} 

\vspace{1cm}
 \textbf{Address:}
Paris 13  University, Institute Galil\'ee, Laboratory of Analysis, Geometry and Applications, CNRS UMR 7539, 95302, 99 avenue J.B Cl\'ement, 93430 Villetaneuse, France  

\texttt{e-mail: duong@math.univ-paris13.fr} 

\end{document}